\documentclass{dpreprint}
\begin{document}

\dtitle{Contractive Spaces and Relatively Contractive Maps}
\dauthor[Darren Creutz]{Darren Creutz}{darren.creutz@vanderbilt.edu}{Vanderbilt University}{}

\datewritten{10 August 2015}

\dabstract{%
\large
We present an exposition of contractive spaces and of relatively contractive maps.  Contractive spaces are the natural opposite of measure-preserving actions and relatively contractive maps the natural opposite of relatively measure-preserving maps.  These concepts play a central role in the work of the author and J.~Peterson on the rigidity of actions of semisimple groups and their lattices and have also appeared in recent work of various other authors.  We present detailed definitions and explore the relationship of these phenomena with other aspects of the ergodic theory of group actions, proving along the way several new results, with an eye towards explaining how contractiveness is intimately connected with rigidity phenomena.
}

\makepreprint

\section{Introduction}

Contractive spaces were introduced by Jaworski \cite{Ja94} originally under the name strongly approximately transitive (SAT) actions and are the natural opposite of measure-preserving actions: if $G$ is a group acting on a probability space $(X,\nu)$ then the action is measure-preserving when for every measurable set $B$ and every $g \in G$ it holds that $\nu(gB) = \nu(B)$; the action is contractive when for every measurable set $B$ with $\nu(B) > 0$ it holds that $\sup_{g} \nu(gB) = 1$.

In the case when the acting group $G$ is amenable, if $G$ acts continuously (or merely Borel) on a compact metric space then a simple compactness argument (in weak*) shows the existence of probability measures preserved by the action.  In contrast, if $G$ is nonamenable then there always exists spaces on which there is no preserved measure.  However, there is a natural way to associate to $G$ a contractive space, the Poisson boundary, that is intimately connected with its action on any compact metric space.

The central idea in the (amenability half of) the rigidity theorem for actions of semisimple groups and their lattices developed in \cite{SZ94} and completed in \cite{CP14} is that if $(X,\nu)$ is any probability space on which $G$ acts measurably and $(B,\beta)$ is any Poisson boundary of $G$ then $(B \times X, \beta \times \nu) \to (X,\nu)$ is a relatively contractive map (defined below).  A crucial feature of relatively contractive maps is the uniqueness theorem (Theorem \ref{T:relcontractiveunique}) and the various factor theorems (Theorems \ref{T:contractiveIFT} and \ref{T:IFT}) that follow from it.  In this sense, contractiveness (and its relativized version) are powerful ideas in the theory of actions of nonamenable groups.

In the context of stationary actions, making use of the Poisson boundary and the fact that the product space is a relatively contractive extension, Furstenberg and Glasner developed the beginnings of a structure theory for non-measure-preserving actions of groups \cite{FG10} (though they did not have the machinery of relatively contractive maps and relied entirely on the contractive nature of the Poisson boundary in their proof).

Our purpose here is to explore the concepts of contractive spaces and relatively contractive maps and their relation to other concepts in the ergodic theory of group actions, notably to relatively measure-preserving maps and to joinings.  Many of the results presented appear (implicitly or explicitly) in \cite{CP14} and \cite{Cre13} but we include some new results.  In particular, we present an example (based on an observation of Glasner and Weiss \cite{glasnerweiss}) of a contractive stationary space that is not a Poisson boundary of the acting group, the first such example known to the author.  The existence of such a space demonstrates the usefulness of studying contractive spaces and relatively contractive maps in their own right, and not merely as a part of the study of boundary actions.

We conclude the paper with a discussion of some of the main applications of relatively contractive maps to rigidity phenomena of lattices.  In particular, we give an outline of how the notions discussed here play a key role in the proofs of results such as:
\begin{theorem*}[Theorems \ref{T:app1} and \ref{T:app2}]
Let $G$ be a semisimple group with trivial center and no compact factors with at least one factor being a connected (real) Lie group with property $(T)$.

Let $\Gamma < G$ be an irreducible lattice (meaning that the projection of $\Gamma$ is dense in every proper normal subgroup of $G$).  Then:

\emph{(i)} every measure-preserving action $\Gamma \actson (X,\nu)$ on a nonatomic probability space is essentially free; and

\emph{(ii)} if $\pi : \Gamma \to \mathcal{U}(M)$ is a representation into the unitary group of a finite factor $M$ such that $\pi(\Gamma)^{\prime\prime} = M$ then either $M$ is finite-dimensional or $\pi$ extends to an isomorphism of the group von Neumann algebra $L\Gamma \simeq M$.
\end{theorem*}

\noindent\textbf{Acknowledgments} The author wishes to thank the referee for numerous helpful suggestions and corrections and, in particular, for noticing an important issue with the definition of relatively contractible spaces and providing a solution.

\section{Group Actions on Probability Spaces}

Throughout the paper $G$ will denote a locally compact second countable topological  group and $\Gamma$ will be reserved for countable discrete groups.  Often $\Gamma < G$ will be a lattice.

\subsection[G-Spaces and G-Maps]{$G$-Spaces and $G$-Maps}

\begin{definition}
A \textbf{$G$-space} is a standard Borel probability space $(X,\nu)$ that is equipped with an action of $G$ such that $\nu$ is quasi-invariant under the action (the class of null sets is preserved by the action).  This will be written $G \actson (X,\nu)$.
\end{definition}

\begin{definition}
Let $G \actson (X,\nu)$ be a $G$-space.  The \textbf{translate of $\nu$ by $g \in G$} is the probability measure $g\nu$ defined by $g\nu(E) = \nu(g^{-1}E)$ for all measurable sets $E$.
\end{definition}

\begin{definition}
Let $(X,\nu)$ and $(Y,\eta)$ be $G$-spaces.  A measurable map $\pi : X \to Y$ such that $\pi_{*}\nu = \eta$ is a \textbf{$G$-map} when $\pi$ is $G$-equivariant: $\pi(gx) = g\pi(x)$ for all $g \in G$ and almost every $x \in X$ (here $\pi_{*}\nu$ is the \textbf{pushforward measure} defined by, for $E$ a measurable subset of $Y$, $\pi_{*}\nu(E) = \nu(\pi^{-1}(E))$).
\end{definition}

\begin{definition}
Let $\pi : (X,\nu) \to (Y,\eta)$ be a $G$-map of $G$-spaces.  The \textbf{disintegration} of $\nu$ over $\eta$ is the almost everywhere unique map $D_{\pi} : Y \to P(X)$ such that the support of $D_{\pi}(y)$ is contained in $\pi^{-1}(y)$ and that $\int_{Y} D_{\pi}(y)~d\eta(y) = \nu$.
\end{definition}

\subsection{Point Realizations}

\begin{definition}
Let $G \actson (X,\nu)$ be a $G$-space.  A \textbf{point realization} (also called a \textbf{compact model}) of the action is a compact metric space $X_{0}$ equipped with a continuous $G$-action $G \actson X_{0}$ and a Borel probability measure (with full support) $\nu_{0}$ such that the action $G \actson (X_{0},\nu_{0})$ is measurably isomorphic to $G \actson (X,\nu)$ (meaning there is a measurable isomorphism defined almost everywhere).
\end{definition}

\begin{definition}
Let $\pi : (X,\nu) \to (Y,\eta)$ be a $G$-map of $G$-spaces.  A \textbf{point realization} of $\pi$ is a continuous map $\pi_{0} : X_{0} \to Y_{0}$ of compact metric spaces such that $G \actson (X_{0},\nu_{0})$ and $G \actson (Y_{0},\eta_{0})$ are point realizations of the $G$-spaces and such that $\pi_{0} : (X_{0},\nu_{0}) \to (Y_{0},\eta_{0})$ is measurably isomorphic to $\pi : (X,\nu) \to (Y,\eta)$ (the obvious diagram of maps commutes in the category of measurable $G$-maps).
\end{definition}

The following result, due to Mackey \cite{Ma62}, states that in the case of locally compact second countable groups (the class we restrict ourselves to), measurable actions can always be realized as actions on points:
\begin{theorem}
When $G$ is locally compact second countable, there exists a point realization of every $G$-map of $G$-spaces.
\end{theorem}

The proof is somewhat technical and the reader is referred to \cite{CP14} or \cite{Cre13} for a detailed proof.

\subsection{Measure-Preserving Spaces and Relatively Measure-Preserving Maps}

\begin{definition}
Let $(X,\nu)$ be a $G$-space.  Then $(X,\nu)$ is \textbf{measure-preserving} when $g\nu = \nu$ for all $g \in G$.
\end{definition}

\begin{definition}
Let $G$ be a locally compact second countable group and $\pi : (X,\nu) \to (Y,\eta)$ a $G$-map of $G$-spaces.  Then $\pi$ is \textbf{relatively measure-preserving} when the disintegration of $\nu$ over $\eta$ via $\pi$, $D_{\pi} : Y \to P(X)$, is $G$-equivariant: $D_{\pi}(gy) = gD_{\pi}(y)$ for all $g \in G$ and almost every $y \in Y$.
\end{definition}

\subsection{Joinings}

Joinings between $G$-spaces are one of the main tools in the ergodic theory of group actions.  The reader is referred to \cite{glasner} for a detailed introduction to the theory of joinings and how they can be used to define a notion of orthogonality (or disjointness) of actions of groups on probability spaces.  In the context of measure-preserving actions, joinings have proven their usefulness in a variety of ways, e.g.~\cite{furstjoinings}, \cite{delJuncoRudolph}, \cite{glasner}, \cite{delarue}.

\begin{definition}
Let $(X,\nu)$ and $(Y,\eta)$ be $G$-spaces.  Let $\alpha \in P(X \times Y)$ such that $(\mathrm{pr}_{X})_{*}\alpha = \nu$, $(\mathrm{pr}_{Y})_{*}\alpha = \eta$ and $\alpha$ is quasi-invariant under the diagonal $G$-action.  The space $(X \times Y, \alpha)$ with the diagonal $G$-action is called a \textbf{joining} of $(X,\nu)$ and $(Y,\eta)$.
\end{definition}

\begin{definition}
A joining $\alpha$ of $G$-spaces is \textbf{$G$-invariant} when $\alpha$ is $G$-measure-preserving under the diagonal action.
\end{definition}

\begin{definition}
Let $(X,\nu)$ and $(Y,\eta)$ be $G$-spaces.  The space $(X\times Y, \nu \times \eta)$ with the diagonal $G$-action is the \textbf{independent joining} of $(X,\nu)$ and $(Y,\eta)$.
\end{definition}

\begin{proposition}
Let $\alpha \in P(X\times Y)$ be a joining of the $G$-spaces $(X,\nu)$ and $(Y,\eta)$.  Consider the projection $p : X \times Y \to Y$.  The disintegration of $\alpha$ over $\eta$ via $p$ is of the form $D_{p}(y) = \alpha_{y} \times \delta_{y}$ for some $\alpha_{y} \in P(X)$ almost surely.
\end{proposition}

\begin{definition}[\cite{glasner} Definition 6.9]
Let $(X,\nu)$, $(Y,\eta)$ and $(Z,\zeta)$ be $G$-spaces and let $\alpha$ be a joining of $(X,\nu)$ and $(Y,\eta)$ and $\beta$ be a joining of $(Y,\nu)$ and $(Z,\zeta)$.  Let $\alpha_{y} \in P(X)$ and $\beta_{y} \in P(Z)$ be the projections of the disintegrations of $\alpha$ and $\beta$ over $\eta$.  The measure $\rho \in P(X \times Z)$ by
\[
\rho = \int_{Y} \alpha_{y} \times \beta_{y}~d\eta(y)
\]
is the \textbf{composition} of $\alpha$ and $\beta$.
\end{definition}

\begin{proposition}[\cite{glasner} Proposition 6.10]\label{P:composejoining}
The composition of two joinings is a joining.  If two joinings are $G$-invariant then so is their composition.
\end{proposition}

\subsection{Relative Joinings}

\begin{definition}
Let $(X,\nu)$ and $(Y,\eta)$ be $G$-spaces with a common $G$-quotient $(Z,\zeta)$, that is a diagram of $G$-maps and $G$-spaces as follows:
\begin{diagram}
		&			&(X,\nu)\\
		&			&\dTo^{\pi} \\
(Y,\eta)	&\rTo^{\varphi}	&(Z,\zeta)
\end{diagram}
Treat $X \times Y$ as a $G$-space with the diagonal action.
A $G$-quasi-invariant Borel probability measure $\rho \in P(X \times Y)$ is a \textbf{relative joining} of $(X,\nu)$ and $(Y,\eta)$ over $(Z,\zeta)$ when the following diagram of $G$-maps commutes:
\begin{diagram}
(X \times Y,\rho)		&\rTo^{p_{X}}	&(X,\nu)\\
\dTo^{p_{Y}}	&			&\dTo^{\pi} \\
(Y,\eta)		&\rTo^{\varphi}	&(Z,\zeta)
\end{diagram}
where $p_{X}$ and $p_{Y}$ are the natural projections from $X \times Y$ to $X$ and $Y$, respectively.
\end{definition}

In general, the product $\nu \times \eta$ is not a relative joining of $(X,\nu)$ and $(Y,\eta)$ over $(Z,\zeta)$ unless $(Z,\zeta)$ is trivial since we require that $\pi \circ p_{X} = \varphi \circ p_{Y}$ almost everywhere.  However, there is a notion of independent joining in the relative case:
\begin{definition}
Let $(X,\nu)$ and $(Y,\eta)$ be $G$-spaces with common $G$-quotient $(Z,\zeta)$.  Let $\pi : (X,\nu) \to (Z,\zeta)$ and $\varphi : (Y,
\eta) \to (Z,\zeta)$ be the quotient maps.  The probability measure $\rho \in P(X \times Y)$ given by
\[
\rho = \int_{Z} D_{\pi}(z) \times D_{\varphi}(z)~d\zeta(z)
\]
is the \textbf{independent relative joining} of $(X,\nu)$ and $(Y,\eta)$ over $(Z,\zeta)$.
\end{definition}

Of course, the independent relative joining is a relative joining.  We also note that the independent joining $\nu \times \eta$ is the independent relative joining over the trivial system.

\begin{proposition}\label{P:irjfactor}
Let $\pi : (X,\nu) \to (Y,\eta)$ be a $G$-map of $G$-spaces.  Then the independent relative joining of $(X,\nu)$ and $(Y,\eta)$ over $(Y,\eta)$ is $G$-isomorphic to $(X,\nu)$.
\end{proposition}
\begin{proof}
The independent relative joining is $(X\times Y,\alpha)$ where
\[
\alpha = \int_{Y} D_{\pi}(y) \times \delta_{y}~d\eta(y).
\]
Let $p : X \times Y \to X$ be the projection to $X$.  Let $\alpha_{x} \in P(X \times Y)$ by $\alpha_{x} = \delta_{x} \times \delta_{\pi(x)}$.  Then
\begin{align*}
\int_{X} \alpha_{x}~d\nu(x) &= \int_{Y} \int_{X} \delta_{x} \times \delta_{\pi(x)}~dD_{\pi}(y)(x)~d\eta(y) \\
&= \int_{Y} \int_{X} \delta_{x} \times \delta_{y}~dD_{\pi}(y)(x)~d\eta(y)\\
&= \int_{Y} D_{\pi}(y) \times \delta_{y}~d\eta(y) = \alpha
\end{align*}
and $\alpha_{x}$ is supported on $p^{-1}(x) = \{ x \} \times Y$.  Therefore $D_{p}(x) = \alpha_{x}$ by uniqueness of disintegration.
Since $\alpha_{x}$ is a point mass, then $p$ is an isomorphism so $(X \times Y, \alpha)$ is isomorphic to $(X,\nu)$.
\end{proof}

\subsection{Lattices and Induced Actions}\label{sub:inducedaction}

A countable subgroup $\Gamma < G$ in a locally compact second countable group is a \textbf{lattice} when it is discrete in the $G$-topology and there exists a finite-Haar-measure fundamental domain for $G / \Gamma$.  The main examples of lattices are the arithmetic points of algebraic groups, for example $\mathbb{SL}_{n}(\mathbb{Z}) < \mathbb{SL}_{n}(\mathbb{R})$.

We recall now the construction of the induced action from a lattice to the ambient group, see, e.g., \cite{Zi84}.  
Let $\Gamma < G$ be a lattice in a locally compact second countable group and let $(X,\nu)$ be a $\Gamma$-space.
Take a fundamental domain $F$ for $G / \Gamma$ such that $e \in F$.  Let $m \in P(F)$ be the Haar measure of $G$ restricted to $F$ and normalized to be a probability measure on $F$.  Define the cocycle $\alpha : G \times F \to \Gamma$ by
$\alpha(g,f) = \gamma$ such that $gf\gamma \in F$
and observe that such a $\gamma$ is unique so this is well-defined.  Note that
$\alpha(gh,f) = \alpha(h,f)\alpha(g,hf\alpha(h,f))$
meaning $\alpha$ is indeed a cocycle.  We also remark that $\alpha(f,e) = e$ for $f \in F$ and that $\alpha(\gamma,e) = \gamma^{-1}$ for $\gamma \in \Gamma$.
Consider now the action of $G$ on $F \times X$ given by
\[
g \cdot (f,x) = (gf\alpha(g,f), \alpha(g,f)^{-1}x)
\]
and observe that the measure $m \times \nu$ is quasi-invariant under this action.  So $(F \times X, m \times \nu)$ is a $G$-space.

Also consider the $\Gamma$-action on $(G \times X, \mathrm{Haar} \times \nu)$ given by
\[
\gamma \cdot (g,x) = (g\gamma^{-1}, \gamma x)
\]
and observe that this is quasi-invariant as well.  Since the $\Gamma$-action on $G / \Gamma$ is proper the space of $\Gamma$-orbits of $G \times X$ under that action is well-defined and we denote it by $G \times_{\Gamma} X$ and write elements as $[g,x]$.  Define a $G$-action on $G \times_{\Gamma} X$ by $h \cdot [g,x] = [hg,x]$.

Define the map $\tau : F \times X \to G \times_{\Gamma} X$ by
$\tau(f,x) = [f,x]$
and the map $\rho : G \times_{\Gamma} X \to F \times X$ by
$\rho([g,x]) = (g\alpha(g,e),\alpha(g,e)^{-1}x)$.
Observe that $\rho$ is well-defined since $\alpha(g\gamma,e) = \gamma^{-1}\alpha(g,e)$.

Clearly, $\tau(\rho([g,x])) = [g,x]$
and
$\rho(\tau(f,x)) = (f,x)$
so these maps invert one another.  Moreover,
\begin{align*}
\tau(g\cdot(f,x))
= [gf\alpha(g,f),\alpha(g,f)^{-1}x] 
= [gf,x] = g \cdot [f,x] = g \cdot \tau(f,x)
\end{align*}
and similarly, $\rho(h\cdot [g,x]) = h \cdot \rho([g,x])$
so $\tau$ and $\rho$ are inverse $G$-isomorphisms of $(F \times X, m \times \nu)$ and $(G\times_{\Gamma} X, \alpha)$ where $\alpha = \tau_{*}(m \times \nu)$.

These isomorphisms show that the construction defined is independent of the fundamental domain chosen and we define the \textbf{induced action to $G$ of $\Gamma \actson (X,\nu)$} to be the $G$-space $(G \times_{\Gamma} X, \alpha)$.

More generally, one can induce a $\Gamma$-map of $\Gamma$-spaces.  Let $\pi : (X,\nu) \to (Y,\eta)$ be a $\Gamma$-map of $\Gamma$-spaces.
Fix a fundamental domain $F$ for $G / \Gamma$ and $m \in P(F)$ as above.  Define the map $\Phi : (F \times X, m \times \nu) \to (F \times Y, m \times \eta)$ by
$\Phi(f,x) = (f,\pi(x))$.  Then
\begin{align*}
\Phi(g \cdot (f,x))
= (gf\alpha(g,f),\pi(\alpha(g,f)^{-1}x)) 
= (gf\alpha(g,f),\alpha(g,f)^{-1}\pi(x)) 
= g \cdot \Phi(f,x)
\end{align*}
so $\Phi$ is a $G$-map of $G$-spaces.  Let $\Pi : (G \times_{\Gamma} X, \alpha) \to (G \times_{\Gamma} Y, \beta)$ be the image of $\Phi$ over the canonical isomorphisms defined above for the induced actions.
The $G$-map $\Pi$ between the induced $G$-spaces is referred to as the \textbf{induced $G$-map from the $\Gamma$-map $\pi$}.

\section{Contractive Spaces}

The study of contractive spaces was initiated by Jaworski \cite{Ja94}, \cite{Ja94} as part of his proof of the Choquet-Deny Theorem  for nilpotent groups.  The key observation in his work is that Poisson boundaries enjoy a property he termed strong approximate transitivity (SAT), a name based on the fact that it can be viewed as stronger version of the approximate transitivity property of Connes and Woods.

Later work indicated that contractive is a better term for this property as it is both more descriptive and removes the somewhat misleading appearance of a connection to the AT property (the reader is referred to \cite{CS14}, \cite{CP14} and \cite{Cre13}).  Contractiveness is the natural opposite of measure-preserving and is orthogonal to measure-preserving in a variety of ways (presented in the sequel).

\subsection{The Definition of Contractive Actions}

\begin{definition}[Jaworski \cite{Ja94}]
An action $G \actson (X,\nu)$ is a \textbf{contractive} when for every measurable set $B$ with $\nu(B) > 0$ it holds that $\sup_{g} \nu(gB) = 1$.
\end{definition}

This is easily seen to be equivalent to the statement that given any positive measure set $B$ there exists a sequence $\{ g_{n} \}$ of group elements such that $\nu(g_{n}B) \to 1$ which can be thought of saying that under $\{ g_{n}^{-1} \}$ the whole space $A$ contracts to the set $B$.

\begin{proposition}
If $G \actson (X,\nu)$ is both contractive and measure-preserving, it is the trivial (one-point) system.
\end{proposition}
\begin{proof}
Let $B$ be any positive measure set.  For any $\epsilon > 0$ there exists $g \in G$ such that $\nu(gB) > 1 - \epsilon$ since the action is contractive.  But $\nu(gB) = \nu(B)$ since the action is measure-preserving.  Therefore $\nu(B) > 1 - \epsilon$ for all $\epsilon$.  We conclude that $\nu(B) = 1$ for any positive measure set meaning that $(X,\nu)$ is (measurably) a single point.
\end{proof}

\begin{proposition}[Jaworski \cite{Ja94}]
If $G \actson (X,\nu)$ is contractive then it is ergodic.
\end{proposition}
\begin{proof}
Let $B$ be a positive measure $G$-invariant set.  Since the action is contractive, for any $\epsilon > 0$ there exists $g \in G$ such that $\nu(gB) > 1 - \epsilon$.  As $B$ is $G$-invariant, $gB = B$ and so $\nu(B) > 1- \epsilon$.  We then conclude that $\nu(B) = 1$ meaning the action is ergodic.
\end{proof}

\subsection{Properties of Contractive Actions}

We now state various equivalent characterizations of contractiveness and properties of such spaces.  These facts are all special cases of results presented in the following section on relatively contractive maps and so we omit proofs here in favor of presenting the more general proofs in the sequel.

\begin{theorem}[Jaworski \cite{Ja94}]
$G \actson (X,\nu)$ is contractive if and only if the map $L^{\infty}(X,\nu) \to L^{\infty}(G,\Haar)$ given by $f \mapsto \widehat{f}$ where
\[
\widehat{f}(g) = g\nu(f) = \int f(gx)~d\nu(x)
\]
is an isometry.
\end{theorem}
Note that for measure-preserving systems, the map $f \mapsto \widehat{f}$ has image precisely equal to the constants and that this is an equivalent characterization of measure-preserving.

\begin{theorem}[Creutz-Shalom \cite{CS14}]
Contractiveness is a property of the measure class: if $G \actson (X,\nu)$ is contractive and $\eta$ is a probability measure on $X$ in the same measure class as $\nu$ then $G \actson (X,\eta)$ is contractive.  In fact, the same sequence of elements in $G$ contracts both measures: if $\nu(g_{n}B) \to 1$ then $\eta(g_{n}B) \to 1$.
\end{theorem}

Viewing contractive actions at the level of point realizations plays a key role in the proof of the uniqueness theorem for relatively contractive maps.  To this end, we present a definition and result due to Furstenberg and Glasner on point realizations of contractive actions:
\begin{definition}[Furstenberg-Glasner \cite{FG10}]
Let $G \actson X$ be a continuous action of a group $G$ on a compact metric space $X$ and let $\nu$ be a Borel probability measure on $X$ with full support.  The action $G \actson (X,\nu)$ is \textbf{contractible} when for every point $x \in X$ there exists a sequence $\{ g_{n} \}$ in $G$ such that $g_{n}\nu \to \delta_{x}$ in weak* (here $\delta_{x}$ is the point mass at $x$).
\end{definition}

\begin{theorem}[Furstenberg-Glasner \cite{FG10}]
A $G$-space $G \actson (X,\nu)$ is contractive if and only if every point realization of the action is contractible.
\end{theorem}

The generalization of the above statement due to the author and J.~Peterson \cite{CP14} is the main technical result in the proof of the uniqueness theorem and factor theorems for relatively contractive maps.  This result also justifies the intuitive idea that contractive spaces have the property that ``the group contracts the measure to every possible point mass".

\subsection{Examples of Contractive Actions}

Poisson boundaries are the main examples of contractive actions.  Introduced by Furstenberg \cite{Fu63} as a means for performing harmonic analysis on Lie groups, Poisson boundaries have led to a variety of deep results in the rigidity theory of lattices in Lie groups, including the celebrated Normal Subgroup Theorem of Margulis \cite{Ma79}.  As our purpose here is to focus on contractive spaces, we include minimal details about Poisson boundaries and refer the reader to the work of Bader and Shalom \cite{BS04} (Section 2) for an excellent overview of the construction of the Poisson boundary and its various properties.

\begin{definition}[Furstenberg \cite{Fu63}]
Let $\Gamma$ be a countable discrete group and $\mu$ a probability measure on $\Gamma$ (an element of $\ell^{1}\Gamma$ that is nonnegative and has norm one) such that the support of $\mu$ generates $\Gamma$ and such that $\mu$ is symmetric: $\mu(\gamma) = \mu(\gamma^{-1})$ for all $\gamma \in \Gamma$.  Consider the random walk on $\Gamma$ with law $\mu$: the space $(\Gamma^{\mathbb{N}},\mu^{\mathbb{N}})$ where $\Gamma$ acts by multiplication on the left of the first element.  The map $T : \Gamma^{\mathbb{N}} \to \Gamma^{\mathbb{N}}$ by $T(\omega_{1},\omega_{2},\ldots) = (\omega_{1}\omega_{2},\omega_{3},\ldots)$ (multiplying the first two elements) commutes with the $\Gamma$-action.  The space of $T$-ergodic components (equivalently, the tail $\sigma$-algebra of the space under the filtration by coordinates) is the \textbf{Poisson boundary} for $\Gamma$ with the measure $\mu$.
\end{definition}

The above definition can be extended in the obvious way to locally compact groups provided the measure $\mu$ on the group $G$ is taken to be nonsingular with respect to the Haar measure (it is enough for some convolution power to be nonsingular).

\begin{theorem}[Jaworksi \cite{Ja94}]
Let $(B,\beta)$ be any Poisson boundary for $G$ (any meaning for any probability measure on $G$).  Then $G \actson (B,\eta)$ is contractive.
\end{theorem}

Since contractiveness is a property of the measure class, if $(B,\beta)$ is a Poisson boundary of $G$ and $\nu$ is a probability measure in the same class as $\beta$ then $G \actson (X,\nu)$ is also contractive.

The Poisson boundary enjoys an additional property, namely that of stationarity: if $G \actson (X,\nu)$ is a $G$-space and $\mu$ is a probability measure on $G$ then the convolution $\mu * \nu$ is the probability measure on $X$ defined by $\mu * \nu(B) = \int_{G} g\nu(B)~d\mu(g)$; a $G$-space is $\mu$-stationary when $\mu * \nu = \nu$.  The fact that Poisson boundaries are stationary follows from the fact that $T_{*}\mu^{\mathbb{N}} = \mu * \mu^{\mathbb{N}}$.  Stationary spaces have received much attention since the same argument showing that amenable groups always have invariant measures can be used to show that if $G$ is any group acting on a compact metric space $X$ then there always exists a stationary measure on $X$.

These facts lead to the natural question of whether there exist contractive spaces admitting a stationary measure in the measure class (for some probability measure on the group with full support) that are not quotients of the Poisson boundary (and also to the question of whether there exist contractive actions admitting no stationary measures at all).  To the best of our knowledge, this question has been unanswered previously and we now present an example of such a space.

\begin{theorem}
Let $G = \mathbb{R} \rtimes \mathbb{R}^{+}$ be the ``$ax + b$ group" consisting of all maps $\mathbb{R} \to \mathbb{R}$ of the form $x \mapsto ax + b$ for constants $a,b \in \mathbb{R}$.  Then the natural action of $G$ on $\mathbb{R}$ (equipped with an probability measure in the class of the Lebesgue measure on $\mathbb{R}$) is contractive but is not a quotient of the Poisson boundary of $G$ for any symmetric measure $\mu$ on $G$ that is nonsingular with respect to the Haar measure.
\end{theorem}
\begin{proof}
That the action is contractive is a consequence of the work of Jaworski \cite{Ja94}.  Glasner and Weiss (\cite{glasnerweiss} Example 3.5) observed that the action is not doubly ergodic in the sense of Kaimanovich \cite{kaidouble}.  However, Kaimanovich showed that the action of a group on any quotient of any Poisson boundary (with respect to a stationary admissible measure) is doubly ergodic.  We therefore conclude that this action is contractive but is not a boundary action.
\end{proof}

\section{Relatively Contractive Maps}

Relatively contractive maps were introduced in \cite{CP14} as a generalization of contractive spaces to quotient maps.  We present here the definition of such maps and state (either with proof or with reference to where proofs can be found) the various properties that make them useful.  The main application of these results are the factor theorems appearing in the following section which in turn are the central result needed in the rigidity theorems presented in the final section of the paper.

\subsection{Conjugates of Disintegration Measures}

The principal notion in formulating the idea of relatively contractive maps is to ``conjugate" the disintegration measures.
For a $G$-map of $G$-spaces $\pi : (X,\nu) \to (Y,\eta)$, the disintegration of $\nu$ over $\eta$ can be summarized as saying that for almost every $y \in Y$ there is a unique measure $D_{\pi}(y) \in P(X)$ such that $D_{\pi}(y)$ is supported on the fiber over $y$ and $\int_{Y} D_{\pi}(y)~d\eta(y) = \nu$.

For $g \in G$ and $y \in Y$, we have that $D_{\pi}(gy)$ is supported on the fiber over $gy$, that is, on $\pi^{-1}(gy) = g \pi^{-1}(y)$, and that for any Borel $B \subseteq X$, we have that $gD_{\pi}(y)(B) = D_{\pi}(y)(g^{-1}B)$
meaning that $gD_{\pi}(y)$ is supported on $g\pi^{-1}(y)$.  Therefore we can formulate the following:
\begin{definition}
Let $\pi : (X,\nu) \to (Y,\rho)$ be a $G$-map of $G$-spaces.
The \textbf{conjugated disintegration measure} over $\pi$ at a point $y \in Y$ by the group element $g \in G$ is
\[
D_{\pi}^{(g)}(y) = g^{-1}D_{\pi}(gy).
\]
\end{definition}

The preceding discussion shows that $D_{\pi}^{(g)}(y)$ is supported on $g^{-1} g \pi^{-1}(y) = \pi^{-1}(y)$.  Hence:
\begin{proposition}\label{P:conjugateddisintegration}
Let $\pi : (X,\nu) \to (Y,\eta)$ be a $G$-map of $G$-spaces and fix $y \in Y$.  The conjugated disintegration measures
\[
\mathcal{D}_{y} = \{ g^{-1}D_{\pi}(gy) : g \in G \}
\]
are all supported on $\pi^{-1}(y)$.
\end{proposition}

Another approach to the conjugates of disintegration measures is to observe that:
\begin{proposition}
Let $\pi : (X,\nu) \to (Y,\eta)$ be a $G$-map of $G$-spaces.  For any $g \in G$ then $\pi : (X, g^{-1}\nu) \to (Y, g^{-1}\eta)$ is also a $G$-map of $G$-spaces.  Let $D_{\pi} : Y \to P(X)$ be the disintegration of $\nu$ over $\eta$.  Then $D_{\pi}^{(g)}$ is the disintegration of $g^{-1}\nu$ over $g^{-1}\eta$.
\end{proposition}

\begin{proof}
To see that $\pi$ maps $(X,g^{-1}\nu)$ to $(Y,g^{-1}\eta)$ follows from $\pi$ being $G$-equivariant.

We have already seen that $g^{-1} D_{\pi}(gy)$ is supported on $\pi^{-1}(y)$ so to prove the proposition it remains only to show that $\int g^{-1} D_{\pi}(gy)~dg^{-1}\eta(y) = g^{-1}\nu$.  This is clear as
\begin{align*}
\int_{Y} g^{-1} D_{\pi}(gy)~dg^{-1}\eta(y) &= g^{-1} \int_{Y} D_{\pi}(g g^{-1} y)~d\eta(y) = g^{-1} \int_{Y} D_{\pi}(y)~d\eta(y) = g^{-1}\nu
\end{align*}
since $D_{\pi}$ disintegrates $\nu$ over $\eta$.
\end{proof}

A basic fact we will need in what follows is that  the conjugated disintegration measures are mutually absolutely continuous to one another (over a fixed point $y$ of course, as $y$ varies they have disjoint supports):
\begin{proposition}
Let $\pi : (X,\nu) \to (Y, \eta)$ be a $G$-map of $G$-spaces.  For almost every $y$ the set
\[
\mathcal{D}_{y} = \{ g^{-1}D_{\pi}(gy) : g \in G \}
\]
is a collection of mutually absolutely continuous probability measures supported on $\pi^{-1}\{ y \}$.
\end{proposition}
\begin{proof}
We will prove this in the case when $G$ is countable, the reader is referred to \cite{CP14} for a proof in the locally compact case.
For $g \in G$ write
\[
A_{g} = \{ y \in Y : D_{\pi}(y)~\text{and}~g^{-1}D_{\pi}(gy) \text{ are not in the same measure class} \}.
\]
Then $A_{g}$ is a Borel set for each $g \in G$ since $D_{\pi} : Y \to P(X)$ is a Borel map and the equivalence relation on $P(X)$ given by $\alpha \sim \beta$ if and only if $\alpha$ and $\beta$ is in the same measure class is Borel.

Since $g^{-1}D_{\pi}(gy)$ is the disintegration of $g^{-1}\nu$ over $g^{-1}\eta$ and $g^{-1}\nu$ is in the same measure class as $\nu$, Lemma \ref{L:abscontmeas} (following the proof) gives that $\eta(A_{g}) = 0$ for each $g \in G$.
Therefore $\eta(\bigcup_{g \in G} A_{g}) = 0$ since $G$ is countable, proving the theorem.
\end{proof}

\begin{lemma}\label{L:abscontmeas}
Let $(X,\nu)$ be a probability space and $\pi : (X,\nu) \to (Y,\pi_{*}\nu)$ a measurable map to a probability space.  Let $\alpha$ be a probability measure in the same measure class as $\nu$.  Let $D(y)$ denote the disintegration of $\nu$ over $\pi_{*}\nu$ via $\pi$ and let $D^{\prime}(y)$ denote the disintegration of $\alpha$ over $\pi_{*}\alpha$ via $\pi$.  Then for almost every $y \in Y$, $D(y)$ and $D^{\prime}(y)$ are in the same measure class.
\end{lemma}
\begin{proof}
Since $\alpha$ and $\nu$ are in the same measure class, the Radon-Nikodym derivative $\frac{d\alpha}{d\nu}$ exists and is in $L^{1}(X,\nu)$.  Let $f_{y}$ be the restriction of $\frac{d\alpha}{d\nu}$ to $\pi^{-1}(y)$.  Then by the uniqueness of the disintegration, $D^{\prime}(y) = f_{y}D(y)$ almost surely.  Therefore $D^{\prime}(y)$ is absolutely continuous with respect to $D(y)$ almost surely.  By a symmetric argument, $D(y)$ is absolutely continuous with respect to $D^{\prime}(y)$ so they are in the same measure class.
\end{proof}

\subsection{The Definition of Relatively Contractive Maps}

\begin{definition}
Let $\pi : (X, \nu) \to (Y, \eta)$ be a $G$-map of $G$-spaces.  We say $\pi$ is \textbf{relatively contractive} when for almost every $y \in Y$ and any measurable $B \subseteq X$ with $D_{\pi}(y)(B) > 0$ it holds that $\sup_{g \in G} D_{\pi}^{(g)}(y)(B) = 1$.
\end{definition}

Just as measure-preserving actions are precisely those actions which are relatively measure-preserving extensions of a point (an easy exercise from the definition of relative measure-preserving), relatively contractive maps generalize contractive spaces:
\begin{proposition}
A $G$-space $(X, \nu)$ is contractive if and only if it is a relatively contractive extension of a point.
\end{proposition}
\begin{proof}
In the case where $(Y, \eta) = 0$ is the trivial one point system, the disintegration measure is always $\nu$ and so being a relatively contractive extension reduces to the definition of contractive: $g^{-1}D_{\pi}(g \cdot 0) = g^{-1}\nu$ for all $g \in G$ since $g \cdot 0 = 0$ and therefore $\sup_{g} D_{\pi}^{(g)}(0)(B) = 1$ implies $\sup_{g} g^{-1}D_{\pi}(0)(B) = 1$ so $\sup_{g} g^{-1}\nu(B) = 1$ for all measurable $B$ with $\nu(B) > 0$.
\end{proof}

We remark that if $\pi : (X,\nu) \to (Y,\eta)$ is a relatively measure-preserving $G$-map of $G$-spaces then the conjugated disintegration measures have the property that $D_{\pi}^{(g)}(y) = D_{\pi}(y)$ for all $g \in G$.  From this, it is easy to obtain the first indication that relatively contractive maps are orthogonal to relatively measure-preserving maps:
\begin{proposition}\label{P:bothrels}
Let $\pi : (X,\nu) \to (Y,\eta)$ be a $G$-map of $G$-spaces that is both relatively measure-preserving and relatively contractive.  Then $\pi$ is an isomorphism.
\end{proposition}
\begin{proof}
Let $B$ be a measurable subset of $X$.  Since $\pi$ is relatively contractive, for almost every $y$ such that $D_{\pi}(y)(B) > 0$ there exists $\{ g_{n} \}$ in $G$ such that $D_{\pi}^{(g_{n})}(y)(B) \to 1$.  Since $\pi$ is relatively measure-preserving, $D_{\pi}^{(g_{n}})(y)(B) = D_{\pi}(y)(B)$.  Therefore $D_{\pi}(y)(B) = 1$ for almost every $y$ such that $D_{\pi}(y)(B) > 0$.  Then
\[
\nu(B) = \int_{Y} D_{\pi}(y)(B)~d\eta(y) = \int_{Y} \bbone_{\pi(B)}(y)~d\eta(y) = \eta(\pi(B)).
\]
As this holds for every measurable set $B$, $\pi$ is an isomorphism.
\end{proof}

\subsection{The Algebraic Characterization}

Generalizing Jaworksi \cite{Ja94}, we may characterize relatively contractive maps algebraically:
\begin{proposition}\label{P:contractiveexte}
Let $\pi : (X, \nu) \to (Y, \rho)$ be a $G$-map of $G$-spaces.  Then $\pi$ is relatively contractive if and only if the map $f \mapsto D_{\pi}^{(g)}(y)(f)$ is an isometry between $L^{\infty}(X, D_{\pi}(y))$ and $L^{\infty}(G, \mathrm{Haar})$ for almost every $y \in Y$ (here $D_{\pi}^{(g)}(y)(f)$ is a function of $g$).
\end{proposition}
\begin{proof}
Assume $\pi$ is relatively contractive.  Take $y$ in the measure one set where the disintegration measures are relatively contractive.  Let $f \in L^{\infty}(X,D_{\pi}(y))$ with $\| f \| = 1$.  Fix $\epsilon > 0$ and let $B$ be a measurable set such that $f(x) > 1 - \epsilon$ for $x \in B$ (replacing $f$ with $-f$ if necessary) and such that $D_{\pi}(y)(B) > 0$.  As $\pi$ is relatively contractive, there exists $g \in G$ such that $D_{\pi}^{(g)}(y)(B) > 1 - \epsilon$.  Then $D_{\pi}^{(g)}(f) > 1 - 2\epsilon$.  As $\epsilon$ was arbitrary, this shows that the map is an isometry on the norm one functions hence is an isometry as claimed.

Conversely, assume the map is an isometry for almost every $y$.  For such a $y$, let $B \subseteq \pi^{-1}(y)$ with $D_{\pi}(y)(B) > 0$ and then
$1 = \| \bbone_{B} \|_{\infty} = \sup_{g} D_{\pi}^{(g)}(y)(B)$
so $\pi$ relatively contractive.
\end{proof}

Note that $\pi$ is relatively measure-preserving if and only if the map that would be isometric for relatively contractive, $f \mapsto D_{\pi}^{(g)}(y)(f)$, is simply the map $f \mapsto D_{\pi}(y)(f)$ which is the projection to the ``constants" on each fiber.

We remark that in effect there is a zero-one law for relatively contractive extensions.  Namely, if $\pi : (X,\nu) \to (Y,\eta)$ is a $G$-map of ergodic $G$-spaces then the set of $y$ such that $D_{\pi}^{(g)}(y)$ induces an isometry $L^{\infty}(X,D_{\pi}(y)) \to L^{\infty}(G,\mathrm{Haar})$ has either measure zero or measure one.  This follows from the fact that the set of such $y$ must be $G$-invariant and hence follows by ergodicity: if $D_{\pi}^{(g)}(y)$ induces an isometry then for any $h \in G$ and $f \in L^{\infty}(X,\nu)$
\[
\sup_{g \in G} \big{|}D_{\pi}^{(g)}(hy)(f)\big{|} = \sup_{g \in G} \big{|}D_{\pi}^{(gh)}(y)(h \cdot f)\big{|}
= \sup_{g \in G} \big{|}D_{\pi}^{(g)}(y)(h \cdot f)\big{|} = \| h \cdot f \| = \| f \|.
\]

\subsection{Relatively Contractible Spaces}

Generalizing the idea of contractible spaces as point realizations of contractive actions, in \cite{CP14} the notion of relatively contractible (point) spaces was introduced and used to characterize relatively contractive maps.

\begin{definition}
Let $\pi : (X,\nu) \to (Y,\eta)$ be a $G$-map of $G$-spaces.  A point realization $\pi_{0} : (X_{0},\nu_{0}) \to (Y_{0},\eta_{0})$ for this map is \textbf{relatively contractible} when for $\eta_{0}$-almost every $y \in Y_{0}$ and every $x$ in the support of $D_{\pi}(y_{0})$ there exists a sequence $g_{n} \in G$ such that $D_{\pi_{0}}^{(g_{n})}(y) \to \delta_{x}$ in weak*.
\end{definition}

This definition gives rise to the intuitive view of relatively contractive maps as those where almost every fiber can be contracted to any point mass under the group action (though it must be kept in mind that the sequence which contracts one fiber to a point mass may not contract the other fibers, this is especially critical when $G$ is uncountable).

\begin{theorem}\label{T:contractiveexttopo}
Let $\pi : (X,\nu) \to (Y,\eta)$ be a $G$-map of $G$-spaces.  Then $\pi$ is relatively contractive if and only if  every continuous compact model of $\pi$ is relatively contractible.
\end{theorem}
\begin{proof}
Assume that $\pi$ is relatively contractive.  By Proposition \ref{P:contractiveexte}, there is a measure one set of $y$ such that $f \mapsto D_{\pi}^{(g)}(y)(f)$ is an isometry from $L^{\infty}(X, D_{\pi}(y))$ and $L^{\infty}(G,\mathrm{Haar})$. Fix $y$ in that set and let $x \in X$ be in the support of $D_{\pi}(y)$.  Choose $f_{n} \in C(X)$ such that $0 \leq f_{n} \leq 1$, $\| f_{n} \| = 1$ and $f_{n} \to \bbone_{\{x\}}$ pointwise (possible since $C(X)$ separates points) and such that $f_{n+1} \leq f_{n}$.  Since $\pi$ is relatively contractive (and $x$ is in the support of $D_{\pi}(y)$), $\sup_{g} D_{\pi}^{(g)}(y)(f_{n}) = 1$ for each $n$.  Choose $g_{n} \in G$ such that
\[
1 - \frac{1}{n} < D_{\pi}^{(g_{n})}(y)(f_{n})
\]
and observe then that, since $f_{n+1} \leq f_{n}$,
\[
1 - \frac{1}{n+1} < D_{\pi}^{(g_{n+1})}(y)(f_{n+1}) \leq D_{\pi}^{(g_{n+1})}(y)(f_{n})
\]
and therefore $\lim_{m\to\infty} D_{\pi}^{(g_{m})}(f_{n}) = 1$ for each fixed $n$.

Now $P(X)$ is compact so there exists a limit point $\zeta \in P(X)$ such that $\zeta = \lim_{j} D_{\pi}^{(g_{n_{j}})}(y)$ along some subsequence.  Now $\zeta(f_{n}) = 1$ for each $n$ by the above and $f_{n} \to \bbone_{\{x\}}$ is pointwise decreasing so by bounded convergence $\zeta(\{ x \}) = \lim \zeta(f_{n}) = 1$.  This means that for almost every $y$, the conclusion holds for all $x \in \pi^{-1}(y)$.

For the converse, first consider any continuous compact model such that for almost every $y \in Y$ and every $x$ in the support of $D_{\pi}(y)$, there exists a sequence $\{ g_{n} \}$ such that $D_{\pi}^{(g_{n})}(y) \to \delta_{x}$.  Let $f \in C(X)$.  Then the supremum of $f$ on $\mathrm{supp}~D_{\pi}(y)$ is attained at some $x \in \pi^{-1}(y)$ since $\mathrm{supp}~D_{\pi}(y)$ is a closed, hence compact, set.  Take $g_{n}$ such that $g_{n}^{-1} D_{\pi}(g_{n}y) \to \delta_{x}$.   Then $g_{n}^{-1}D_{\pi}(g_{n}y)(f) \to f(x) = \| f \|_{L^{\infty}(D_{\pi}(y))}$.  Hence for $f \in C(X)$ the map is an isometry.

Now assume that for every continuous compact model for $\pi$ and for almost every $y$ and every $x \in \mathrm{supp}~D_{\pi}(y)$ there is a sequence $g_{n} \in G$ such that $g_{n}^{-1}D_{\pi}(g_{n}y) \to \delta_{x}$.

Suppose that $\pi$ is not relatively contractive.  Then there exists a measurable set $A \subseteq X$ with $\nu(A) > 0$ and $1 > \delta > 0$ such that 
\[
B = \{ y \in Y : D_{\pi}(y)(A) > 0 \text{ and } \sup_{g} D_{\pi}^{(g)}(y)(A) \leq 1 - \delta \} > 0
\]
has $\eta(B) > 0$.

Fix $\epsilon > 0$.  Let $\psi_{n} \in C_{c}(G)$ be an approximate identity ($\psi_{n}$ are nonnegative continuous functions with $\int \psi_{n} dm = 1$ where $m$ is a Haar measure on $G$ such that the compact supports of the $\psi_{n}$ are a decreasing sequence and $\cap_{n} \mathrm{supp}~\psi_{n} = \{ e \}$; the reader is referred to \cite{FG10} Corollary 8.7).  Define $f_{n} = \bbone_{A} * \psi_{n} = \int_{G} \bbone_{A}(hx) \psi_{n}(h)~dm(h)$.  Then the $f_{n}$ are $G$-continuous functions by \cite{FG10} Lemma 8.6.

By Proposition \ref{P:fghelp} (below),
\[
\lim_{n} \| \bbone_{A} * \psi_{n} \|_{L^{\infty}(X,D_{\pi}(y))} = 1
\]
for all $y \in B$.

There then exists a set $B_{1} \subseteq B$ with $\eta(B_{1}) > \eta(B) - \epsilon$ and $N \in \mathbb{N}$ such that for all $y \in B_{1}$ and all $n \geq N$, $\| \bbone_{A} * \psi_{n} \|_{L^{\infty}(X,D_{\pi}(y))} > 1 - \epsilon$.
Let $V$ be a compact set neighborhood of the identity in $G$ such that $| \eta(B_{1} \cap h^{-1}B_{1}) - \eta(B_{1})| < \epsilon$ for all $h \in V$ (possible as the $G$-action is continuous on the algebra of measurable sets).  Choose $n \geq N$ such that the support of $\psi = \psi_{n}$ is contained in $V$.

Set $f = \bbone_{A} * \psi$.
Since $f$ is $G$-continuous there exists a continuous compact model on which $f \in C(X)$ by \cite{FG10} Theorem 8.5.  Hence, for almost every $y \in Y$,
\[
\sup_{g} D_{\pi}^{(g)}(y)(f) = \| f \|_{L^{\infty}(X,D_{\pi}(y))}.
\]
Removing a null set from $B_{1}$, then for all $y \in B_{1}$ there exists $g_{y} \in G$ such that
\[
D_{\pi}^{(g_{y})}(y)(f) > \| f \|_{L^{\infty}(X,D_{\pi}(y))} - \epsilon > 1 - 2\epsilon.
\]
Observe that
{\allowdisplaybreaks
\begin{align*}
(1 - 2\epsilon)\eta(B_{1}) &\leq \int_{B_{1}} D_{\pi}^{(g_{y})}(f)~d\eta(y) \\
&= \int_{B_{1}} \int_{X} f(g_{y}^{-1}x)~dD_{\pi}(g_{y}y)(x)~d\eta(y) \\
&= \int_{B_{1}} \int_{X} \int_{G} \bbone_{A}(hg_{y}^{-1}x) \psi(h)~dm(h)~dD_{\pi}(g_{y}y)~d\eta(y) \\
&= \int_{G} \int_{B_{1}} D_{\pi}(g_{y}y)(g_{y}h^{-1}A)~d\eta(y) \psi(h)~dm(h) \\
&= \int_{G} \int_{hB_{1}} D_{\pi}(g_{y}h^{-1}y)(g_{y}h^{-1}A)~dh\eta(y) \psi(h)~dm(h) \\
&= \int_{G} \int_{hB_{1}} D_{\pi}^{(g_{y}h^{-1})}(y)(A)~dh\eta(y) \psi(h)~dm(h) \\
&\leq \int_{G} \int_{hB_{1}} \sup_{g} D_{\pi}^{(g)}(y)(A)~dh\eta(y) \psi(h)~dm(h) \\
&= \int_{G} \Big{(} \int_{hB_{1} \setminus B_{1}} \sup_{g} D_{\pi}^{(g)}(y)(A)~dh\eta(y) \\
&\quad\quad + \int_{hB_{1} \cap B_{1}} \sup_{g} D_{\pi}^{(g)}(y)(A)~dh\eta(y) \Big{)} \psi(h)~dm(h) \\
&\leq \int_{G} \big{(} h\eta(hB_{1} \setminus B_{1}) + (1 - \delta)h\eta(hB_{1} \cap B_{1}) \big{)} \psi(h)~dm(h) \\
&= \int_{G} \big{(} h\eta(hB_{1}) - \delta h\eta(hB_{1} \cap B_{1}) \big{)} \psi(h)~dm(h) \\
&= \eta(B_{1}) - \delta \int_{G} \eta(B_{1} \cap h^{-1}B_{1}) \psi(h)~dm(h).
\end{align*}}
Now the support of $\psi$ is contained in $V$ and $| \eta(B_{1} \cap h^{-1}B_{1}) - \eta(B_{1})| < \epsilon$ for all $h \in V$.  Therefore
\[
-2\epsilon\eta(B_{1}) \leq - \delta \int_{G} (\eta(B_{1}) - \epsilon) \psi(h)~dm(h) = - \delta\eta(B_{1}) + \delta\epsilon.
\]
Hence
\[
\delta \eta(B_{1}) \leq \epsilon(2\eta(B_{1}) + \delta).
\]
Then
\[
\delta \eta(B) \leq \delta (\eta(B_{1}) + \epsilon) \leq 2\epsilon(\eta(B_{1}) + \delta) \leq 2\epsilon(\eta(B) + \delta).
\]
Since $\delta$ is fixed and this holds for all $\epsilon > 0$, $\eta(B) = 0$ contradicting that $\pi$ is not relatively contractive.
\end{proof}

We remark briefly on how to construct the approximate identity used in the proof.  Let $F : [0,\infty) \to [0,1]$ be a continuous monotone decreasing function such that $F(0) = 1$ and $F(t) = 0$ for $t \geq 1$.  Such functions are easily constructed.  For each $n \in \mathbb{N}$ define $F_{n}(t) = F(nt)$.  Then $F_{n} : [0,\infty) \to [0,1]$ is continuous and $F_{n}(0) = 1$ and $F_{n}(t) = 0$ for $t \geq \frac{1}{n}$.  Therefore $F_{n}(t) \to 0$ for $t > 0$ and $F_{n}(0) \to 1$.  Also, since $F$ is decreasing, $F_{n+1}(t) = F((n+1)t) \leq F(nt) = F_{n}(t)$.  Now returning to our compact metric space $(X,d)$, fix $x_{0} \in X$ and define $f_{n}(x) = F_{n}(d(x,x_{0}))$.  Then $f_{n} \in C(X)$ since $d(\cdot,x_{0}) \in C(X)$ and $F_{n} \in C([0,\infty))$.  Clearly $f_{n+1}(x) \leq f_{n}(x)$ and $f_{n}(x) \to 0$ for $x \ne x_{0}$ and $f_{n}(x_{0}) \to 1$.  This sequence $\{ f_{n} \}$ is then the approximate identity used in the proof.

The following fact was used in the above proof and a detailed proof can be found in \cite{CP14} so we opt to omit it here.
\begin{proposition}\label{P:fghelp}
Let $\pi : (X,\nu) \to (Y,\eta)$ be a $G$-map of $G$-spaces.  Let $\psi_{n} \in C_{c}(G)$ be an approximate identity (the $\psi_{n}$ are nonnegative continuous functions with decreasing compact supports $V_{n}$ such that $\cap V_{n} = \{ e \}$ and $\int \psi_{n} dm = 1$ for $m$ a Haar measure on $G$).  Then for any measurable set $A \subseteq X$ and almost every $y \in Y$ such that $D_{\pi}(y)(A) > 0$,
\[
\lim_{n} \| \bbone_{A} * \psi_{n} \|_{L^{\infty}(X,D_{\pi}(y))} = 1.
\]
\end{proposition}

\subsection{Relatively Contractive Maps and Dense Subgroups}

In general, the map $g \mapsto D_{\pi}^{(g)}(y)$ is not continuous (however, it can be shown to be continuous almost everywhere for almost every $y$) which can be seen by considering an induced action from a lattice to a locally compact second countable group (for such an action, there cannot exist point realizations making the map continuous, a fact left to the reader).  This fact accounts for the difficulty in the proof of the following statement.

\begin{theorem}\label{T:relcontdense}
Let $\pi : (X,\nu) \to (Y,\eta)$ be a relatively contractive $G$-map of $G$-spaces.  Let $G_{0}$ be a countable dense subgroup of $G$.  Then $\pi$ is a relatively contractive $G_{0}$-map.
\end{theorem}
\begin{proof}
Suppose that $\pi$ is not $G_{0}$-relatively contractive.  By the proof of Theorem \ref{T:contractiveexttopo}, there then exists a continuous compact model for $\pi : X \to Y$, a positive measure set $A \subseteq Y$, a nonnegative continuous function $f \in C(X)$ and $\delta > 0$ such that for all $y \in A$,
\[
\sup_{g_{0} \in G_{0}} D_{\pi}^{(g_{0})}(y)(f) \leq \| f \|_{L^{\infty}(X,D_{\pi}(y))} - \delta.
\]

Let $\epsilon > 0$ such that $\eta(A) > \epsilon$.  Since $\pi$ is $G$-relatively contractive, there is a conull Borel set $Y_{00}$ such that for every $y \in Y_{00}$, $\sup_{g} D_{\pi}^{(g)}(y)(f) = \| f \|_{L^{\infty}(X,D_{\pi}(y))}$.

Consider the set
\[
E = \{ (g,y) \in G \times Y_{00} : D_{\pi}^{(g)}(y)(f) \geq \| f \|_{L^{\infty}(X,D_{\pi}(y))} - \epsilon \}.
\]
Since $D_{\pi}$ is a Borel map, this is a Borel set.  By the von Neumann Selection Theorem \cite{vonneumann49} there then exists a conull Borel set $Y_{0} \subseteq Y_{00}$ such that the map $\mathrm{pr}_{Y} : E \to Y_{00}$ admits a Borel section on $Y_{0}$.
Choose a Borel section $g_{y} \in G$ for $y \in Y_{0}$ such that $D_{\pi}^{(g_{y})}(y) \geq \| f \|_{L^{\infty}(X,D_{\pi}(y))} - \epsilon$.

Consider the Borel function $Y \to P(X)$ given by $y \mapsto D_{\pi}(g_{y}y)$.  By Lusin's Theorem, there exists a measurable set $D \subseteq Y$ with $\eta(D) > 1 - \epsilon$ and a continuous map $F : Y \to P(X)$ such that $F(y) = D_{\pi}(g_{y}y)$ for $y \in D$.

For $y \in Y_{0}$, choose $\{ g_{n} \}$ in $G_{0}$ such that $g_{n} \to g_{y}$.
Then $\| g_{y} \cdot f - g_{n} \cdot f \|_{\infty} \to 0$ since $G$ acts continuously on $C(X)$ and $F(g_{y}^{-1}g_{n}y) \to F(y)$ in weak* hence $F(g_{y}^{-1}g_{n}y)(g_{y} \cdot f) \to F(y)(g_{y} \cdot y)$.  Therefore
\begin{align*}
\big{|} F&(y)(g_{y} \cdot f) - F(g_{y}^{-1}g_{n}y)(g_{n} \cdot f) \big{|} \\
&\leq \big{|} F(y)(g_{y} \cdot f) - F(g_{y}^{-1}g_{n}y)(g_{y} \cdot f) \big{|}
+ \big{|} F(g_{y}^{-1}g_{n}y)(g_{y} \cdot f) - F(g_{y}^{-1}g_{n}y)(g_{n} \cdot f) \big{|} \\
&\leq \big{|} F(y)(g_{y} \cdot f) - F(g_{y}^{-1}g_{n}y)(g_{y} \cdot f) \big{|}
+ \int_{X} \big{|} f(g_{y}^{-1}x) - f(g_{n}^{-1}x) \big{|}~dF(g_{y}^{-1}g_{n}y)(x) \\
&\leq \big{|} F(y)(g_{y} \cdot f) - F(g_{y}^{-1}g_{n}y)(g_{y} \cdot f) \big{|}
+ \| g_{y} \cdot f - g_{n} \cdot f \|_{\infty} \to 0.
\end{align*}

Observe that for $y \in D$,
\[
F(y)(g_{y} \cdot f) = D_{\pi}(g_{y}y)(g_{y} \cdot f) = D_{\pi}^{(g_{y})}(y)(f) > \| f \|_{L^{\infty}(X,D_{\pi}(y))} - \epsilon.
\]

Consider the set $D_{n}^{\prime} = \{ y \in A : g_{y}^{-1}g_{n}y \in D \}$.
Then for $y \in D_{n}^{\prime}$,
\[
F(g_{y}^{-1}g_{n}y)(g_{n} \cdot f) = D_{\pi}(g_{y}g_{y}^{-1}g_{n}y)(g_{n} \cdot f) = D_{\pi}^{(g_{n})}(y)(f) \leq \| f \|_{L^{\infty}(X,D_{\pi}(y))} - \delta.
\]
Consider the sets $E_{n} = D \cap D_{n}^{\prime}$.  Since $g_{n} \to g_{y}$, $\eta(E_{n}) \to \eta(D \cap A) > 0$.  For $y \in E_{n}$,
\[
\big{|} F(y)(g_{y} \cdot f) - F(g_{y}^{-1}g_{n}y)(g_{n} \cdot f) \big{|} \geq \delta - \epsilon.
\]
But $\big{|} F(y)(g_{y} \cdot f) - F(g_{y}^{-1}g_{n}y)(g_{n} \cdot f) \big{|} \to 0$ as $n \to \infty$ for every $y$.  This contradiction means that $\pi$ is relatively contractive for $G_{0}$.
\end{proof}

\subsection{Examples of Relatively Contractive Maps}

Let $(X,\nu)$ and $(Y,\eta)$ be contractive $G$-spaces.  In general it need not hold that $(X \times Y, \nu \times \eta)$ is contractive (with the diagonal $G$-action), however:
\begin{example}\label{E:examplecont}
Let $(X,\nu)$ be a contractive $G$-space and $(Y,\eta)$ be a $G$-space.  The map $\mathrm{pr_{Y}} : (X \times Y, \nu \times \eta) \to (Y,\eta)$ is relatively contractive ($X\times Y$ has the diagonal $G$-action).
\end{example}
\begin{proof}
The disintegration measures $D_{\pi}(y)$ are supported on $X \times \delta_{y}$ and have the form $D_{\pi}(y) = \nu \times \delta_{y}$.  Clearly
\[
D_{\pi}^{(g)}(y) = g^{-1} (\nu \times \delta_{gy}) = g^{-1}\nu \times \delta_{y}
\]
and since $(X,\nu)$ is contractive then $\pi$ is relatively contractive.
\end{proof}

More generally, the following holds:
\begin{example}
Let $\pi : (X,\nu) \to (Y,\eta)$ be a relatively contractive $G$-map of $G$-spaces.  Let $(Z,\zeta)$ be a $G$-space.  The map $\pi \times \mathrm{id} : (X \times Z, \nu \times \zeta) \to (Y \times Z, \eta \times \zeta)$ is relatively contractive (where $X\times Z$ and $Y\times Z$ have the diagonal $G$-action).
\end{example}
\begin{proof}
Since the disintegration of the identity is point masses, for almost every $(y,z) \in Y \times X$, it holds that
$D_{\pi \times \mathrm{id}}^{(g)}(y,z) = D_{\pi}^{(g)}(y) \times \delta_{z}$.
Then $\pi$ being relatively contractive implies $\pi \times \mathrm{id}$ is relatively contractive.
\end{proof}

Inducing contractive actions also gives rise to relatively contractive maps:
\begin{theorem}
Let $\Gamma < G$ be a lattice in a locally compact second countable group.  Let $(X,\nu)$ be a contractive $\Gamma$-space and $p : G \times_{\Gamma} X \to G / \Gamma$ be the $G$-map that is the natural projection from the induced $G$-space over $(X,\nu)$ to $G / \Gamma$.  Then $p$ is a relatively contractive $G$-map.
\end{theorem}
\begin{proof}
Treat $G \times_{\Gamma} X$ as $(F \times X, m \times \nu)$ for $F$ a fundamental domain for $G / \Gamma$ with cocycle $\alpha : G \times F \to \Gamma$.  Consider $p : F \times X \to F$ the projection.  The disintegration $D_{p}(f)$ of $m \times \nu$ over $m$ is of the form $D_{p}(f) = \delta_{f} \times \nu$.   For $g \in G$,
\[
D_{p}^{(g)}(f) = g^{-1}D_{p}(gf\alpha(g,f)) = g^{-1}(\delta_{gf\alpha(g,f)} \times \nu) = \delta_{f} \times \alpha(g,f)\nu.
\]
Fix $(f_{0},x_{0}) \in F \times X$ and choose $\gamma_{n} \in \Gamma$ such that $\gamma_{n}\nu \to \delta_{x_{0}}$.  Set $g_{n} = \gamma_{n}^{-1}f_{0}^{-1}$.  Then $\alpha(g_{n},f_{0}) = \gamma_{n}$ so
$D_{p}^{(g_{n})}(f_{0}) = \delta_{f_{0}} \times \gamma_{n}\nu \to \delta_{f_{0}} \times \delta_{x_{0}}$
meaning $p$ is relatively contractive.
\end{proof}

\subsection{Factorization of Contractive Maps}

We now present the first step in the uniqueness theorem for relatively contractive maps which describes the possible intermediate quotient maps.
\begin{theorem}\label{T:contractivecomp}
Let $\pi : (X,\nu) \to (Y,\eta)$ and $\varphi : (Y,\eta) \to (Z,\rho)$ be $G$-maps of $G$-spaces.  If $\varphi \circ \pi$ is relatively contractive then both $\varphi$ and $\pi$ are relatively contractive.
\end{theorem}
\begin{proof}
We use Theorem \ref{T:contractiveexttopo} and take a continuous compact model for $\pi$ to do so.  First observe, for all $g \in G$ and almost every $z$, that $\pi_{*}D_{\varphi\circ\pi}^{(g)}(z) = D_{\varphi}^{(g)}(z)$.  For such $z$ where also $\overline{\mathrm{conv}}~\{ D_{\varphi\circ\pi}^{(g)}(z) \} = P((\varphi\circ\pi)^{-1}(z))$ and every $x$ in the support of $D_{\varphi\circ\pi}(z)$, there is $g_{n} \in G$ such that $D_{\varphi\circ\pi}^{(g_{n})}(z) \to \delta_{x}$.  Therefore
\[
D_{\varphi}^{(g_{n})}(z) = \pi_{*}D_{\varphi\circ\pi}^{(g_{n})}(z) \to \pi_{*}\delta_{x} = \delta_{\pi(x)}
\]
and so for every $y$ in the support of $D_{\varphi}(z)$, the point mass $\delta_{y}$ is a limit point of $D_{\varphi}^{(g)}(z)$.  Hence $\varphi$ is relatively contractive.

Suppose that $\pi$ is not relatively contractive.  Then, by the proof of Theorem \ref{T:contractiveexttopo}, there exists a continuous compact model for $\pi : X \to Y$ such that $f \mapsto |D_{\pi}^{(g)}(y)(f)|$ is not an isometry from $C(X)$ to $L^{\infty}(G)$ for a positive measure set of $y \in Y$.  

Observe that if the map is an isometry on a countable dense set $C_{0} \subseteq C(X)$ then for any $f \in C(X)$ there exists $f_{n} \in C_{0}$ with $f_{n} \to f$ in sup norm, hence
\[
|D_{\pi}^{(g)}(y)(f)| = |D_{\pi}^{(g)}(y)(f - f_{n}) + D_{\pi}^{(g)}(f_{n})| \geq |D_{\pi}^{(g)}(y)(f_{n})| - \| f - f_{n} \|_{\infty}.
\]
For $\epsilon > 0$, choose $n$ such that $\| f - f_{n} \|_{\infty} < \epsilon$.  Choose $g$ such that $|D_{\pi}^{(g)}(y)(f_{n})| > \| f_{n} \| - \epsilon$.  Then
\[
|D_{\pi}^{(g)}(y)(f)| > \| f_{n} \| - \epsilon - \epsilon > \| f \| - 3\epsilon
\]
and so the map is an isometry for $f$ as well.

Therefore, there is a positive measure set of $y$ such that $f \mapsto |D_{\pi}^{(g)}(y)(f)|$ is not an isometry on $C_{0}$.  Hence, since $C_{0}$ is countable, there is some $f \in C_{0}$ and a positive measure set of $y$ such that $\sup_{g} |D_{\pi}^{(g)}(y)(f)| < \| f \|_{L^{\infty}(X,D_{\pi}(y))}$.  So there is some $\delta > 0$ and a measurable set $A \subseteq Y$ with $\eta(A) > 0$ such that $\sup_{g} |D_{\pi}^{(g)}(y)(f)| < \| f \|_{L^{\infty}(X,D_{\pi}(y))} - \delta$ for all $y \in A$.  We may assume (by taking a subset) that $A$ is closed.
Since $\eta$ is a Borel measure, it is regular, hence we may assume $A$ is closed (by taking a subset).

Now there exists a positive measure set $B \subseteq Z$ on which $D_{\varphi}(z)(A) > 0$ for $z \in B$.  For $z \in B$ such that $z$ is in the measure one set on which $\varphi \circ \pi$ contracts to point masses,
\begin{align*}
D_{\varphi \circ \pi}^{(g)}(z)(f)
&= \int_{\varphi^{-1}(z)} D_{\pi}^{(g)}(y)(f)~dD_{\varphi}^{(g)}(z)(y) \\
&= \int_{\varphi^{-1}(z) \cap A} D_{\pi}^{(g)}(y)(f)~dD_{\varphi}^{(g)}(z)(y) + \int_{\varphi^{-1}(z) \setminus A} D_{\pi}^{(g)}(y)(f)~dD_{\varphi}^{(g)}(z)(y) \\
&\leq \int_{\varphi^{-1}(z) \cap A} \| f \|_{L^{\infty}(X,D_{\pi}(y))} - \delta~dD_{\varphi}^{(g)}(z)(y)
+ \int_{\varphi^{-1}(z) \setminus A} \| f \|_{L^{\infty}(X,D_{\pi}(y))}~dD_{\varphi}^{(g)}(z)(y) \\
&\leq \| f \|_{L^{\infty}(X,D_{\varphi \circ \pi}(z))} - \delta D_{\varphi}^{(g)}(z)(A).
\end{align*}

Now for any $x$ in the support of $D_{\varphi \circ \pi}(z)$, there exists $g_{n}$ such that $D_{\varphi \circ \pi}^{(g_{n})}(z) \to \delta_{x}$.  Hence also $D_{\varphi}^{(g_{n})}(z) \to \delta_{\pi(x)}$.  Choose $x \in \pi^{-1}(A) \cap (\mathrm{supp}~D_{\varphi\circ\pi}(z))$ such that $f(x) = \| f \|_{L^{\infty}(X,D_{\varphi \circ \pi}(z))}$ (possible since $\pi^{-1}(A) \cap (\mathrm{supp}~D_{\varphi\circ\pi}(z))$ is closed, hence compact, and $f$ is continuous).
Then
\begin{align*}
f(x) &= \lim_{n} D_{\varphi \circ \pi}^{(g_{n})}(z)(f)
\leq \lim_{n}  \| f \|_{L^{\infty}(X,D_{\varphi \circ \pi}(z))} - \delta D_{\varphi}^{(g_{n})}(z)(A) \\
&= \| f \|_{L^{\infty}(X,D_{\varphi \circ \pi}(z))} - \delta \delta_{\pi(x)}(A) = \| f \|_{L^{\infty}(X,D_{\varphi \circ \pi}(z))} - \delta
\end{align*}
is a contradiction.  Hence $\pi$ is relatively contractive.
\end{proof}

The above statement is the analogue of one direction of the similar well-known fact about relative measure-preserving:
\begin{theorem}
Let $\pi : (X,\nu) \to (Y,\eta)$ and $\varphi : (Y,\eta) \to (Z,\zeta)$ be $G$-maps of $G$-spaces such that $\varphi \circ \pi : (X,\nu) \to (Z,\zeta)$ is relatively measure-preserving.  Then $\pi$ and $\varphi$ are both relatively measure-preserving.  Conversely, if $\pi$ and $\varphi$ are relatively measure-preserving then so is $\varphi \circ \pi$.
\end{theorem}

\begin{corollary}\label{C:factorsstaycontractive}
Any $G$-factor of a contractive $G$-space is a contractive $G$-space.  Any $G$-factor of a measure-preserving $G$-space is a measure-preserving $G$-space.
\end{corollary}
\begin{proof}
Let $(X,\nu)$ be a contractive $G$-space and $\pi : (X,\nu) \to (Y,\eta)$ be a $G$-map of $G$-spaces.  Take $\varphi : (Y,\eta) \to 0$ to be the $G$-map to the trivial one-point space.  Then $\varphi \circ \pi : (X,\nu) \to 0$ is relatively contractive since $(X,\nu)$ is contractive and therefore $\varphi$ is relatively contractive since its composition with $\pi$ is and so $(Y,\eta)$ is contractive.  The same argument applied to relative measure-preserving maps shows the second statement.
\end{proof}

\begin{corollary}\label{C:orthomaps}
Let $(X,\nu)$ be a $G$-space such that $\pi : (X,\nu) \to (Y,\eta)$ is a relatively contractive $G$-map of $G$-spaces and $\varphi : (X,\nu) \to (Z,\zeta)$ is a relatively measure-preserving $G$-map of $G$-spaces.  Then $\pi \times \varphi : (X,\nu) \to (Y \times Z, (\pi \times \varphi)_{*}\nu)$ by $(\pi \times \varphi)(x) = (\pi(x),\varphi(x))$ is a $G$-isomorphism.
\end{corollary} 
\begin{proof}
Consider the $G$-map $\mathrm{pr}_{Y} \circ (\pi \times \varphi) = \pi$.  Since $\pi$ is relatively contractive then both the projection map to $Y$ and $\pi \times \varphi$ are relatively contractive (Theorem \ref{T:contractivecomp}).  Likewise $\mathrm{pr}_{Z} \circ (\pi \times \varphi) = \varphi$ is relatively measure-preserving so the projection to $Z$ and $\pi \times \varphi$ are relatively measure-preserving.  By Proposition \ref{P:bothrels} then $\pi \times \varphi$ is an isomorphism.
\end{proof}

\subsection{Uniqueness of Relatively Contractive Maps}

We are now in a position to present the uniqueness theorem for relatively contractive maps.  The proof is somewhat technical and can be found in \cite{CP14}, we opt to omit it in the interest of brevity.

\begin{theorem}\label{T:relcontractiveunique}
Let $(X,\nu)$ be a contractive $G$-space and $(Y,\eta)$ be a measure-preserving $G$-space.  Let $\psi : (X \times Y, \nu \times \eta) \to (Y,\eta)$ be the natural projection map (treating $(X \times Y, \nu \times \eta)$ as $G$-space with the diagonal action).  Let $\pi : (X \times Y, \nu \times \eta) \to (Z,\alpha)$ be a $G$-map of $G$-spaces and let $\pi^{\prime} : (X \times Y, \nu \times \eta) \to (Z,\alpha^{\prime})$ be a $G$-map of $G$-spaces such that $\alpha^{\prime}$ is in the same measure class as $\alpha$.  Let $\varphi : (Z,\alpha) \to (Y,\eta)$ and $\varphi^{\prime} : (Z,\alpha^{\prime}) \to (Y,\eta)$ be $G$-maps such that $\varphi \circ \pi = \psi$ and $\varphi^{\prime} \circ \pi^{\prime} = \psi$.  That is, we consider the following commutative diagram of $G$-maps and $G$-spaces:
\begin{diagram}
&			&				&(X \times Y, \nu \times \eta)		&					&				\\
&			&\ldTo^{\pi}		&							&\rdTo^{\pi^{\prime}}		&				\\
&(Z,\alpha)	&				&\dTo_{\mathrm{proj}}^{\psi}					&					&(Z,\alpha^{\prime})	\\
&			&\rdTo_{\varphi}	&							&\ldTo_{\varphi^{\prime}}	&				\\
&			&				&(Y,\eta)						&					&
\end{diagram}
Assume that the disintegrations $D_{\varphi}(y)$ of $\alpha$ over $\eta$ via $\varphi$ and the disintegrations $D_{\varphi^{\prime}}(y)$ of $\alpha^{\prime}$ over $\eta$ via $\varphi^{\prime}$ have the property that $D_{\varphi}(y)$ and $D_{\varphi^{\prime}}(y)$ are in the same measure class almost surely.  Then $\pi = \pi^{\prime}$ almost everywhere, $\varphi = \varphi^{\prime}$ almost everywhere and $\alpha = \alpha^{\prime}$.
\end{theorem}

\begin{corollary}[Creutz-Shalom \cite{CS14}]\label{C:contractiveuniqueCS}
Let $(X,\nu)$ be a contractive $G$-space and let $\pi : (X,\nu) \to (Z,\alpha)$ and $\pi^{\prime} : (X,\nu) \to (Z,\alpha^{\prime})$ be $G$-maps of $G$-spaces such that $\alpha$ and $\alpha^{\prime}$ are in the same measure class.  Then $\pi = \pi^{\prime}$ almost surely and $\alpha = \alpha^{\prime}$.
\end{corollary}
\begin{proof}
Consider the composition of maps $\varphi \circ \pi : (X,\nu) \to 0$ where $\varphi : (Z,\eta) \to 0$ is the map to the trivial system.  Since $(X,\nu)$ is contractive, the preceding theorem gives the result.
\end{proof}

\subsection{Joinings With Contractive Spaces}

\begin{theorem}\label{T:contractivejoin}
Let $(X,\nu)$ be a contractive $G$-space and let $(Y,\eta)$ be a $G$-space.  Then there is at most one joining $(X \times Y, \alpha)$ of $(X,\nu)$ and $(Y,\eta)$ such that the projection to $X$ is relatively measure-preserving.
\end{theorem}
\begin{proof}
Let $f \in L^{\infty}(Y,\eta)$ and define 
\[
F(x) = D_{\mathrm{pr}_{X}}(x)(f \circ \mathrm{pr}_{Y}).
\]
Taking compact models for $X$ and $Y$ such that $\pi$ is continuous makes clear that $F$ is a bounded Borel function on $X$.
Then for any $g \in G$ we have that, using that $\mathrm{pr}_{X}$ is relatively measure-preserving,
\begin{align*}
g\nu(F) &= \int_{X} F(gx)~d\nu(x) \\
&= \int_{X} \int_{X \times Y} f(\mathrm{pr}_{Y}(z,y))~dD_{\mathrm{pr}_{X}}(gx)(z,y)~d\nu(x) \\
&= \int_{X} \int_{X \times Y} f(g \mathrm{pr}_{Y}(z,y))~dD_{\mathrm{pr}_{X}}(x)(z,y)~d\nu(x) \\
&= \int_{X \times Y} f(g \mathrm{pr}_{Y}(z,y))~d\alpha(z,y) \\
&= \int_{Y} f(gy)~d(\mathrm{pr}_{Y})_{*}\alpha(y) = \int_{Y} f(gy)~d\eta(y) = g\eta(f).
\end{align*}

Suppose now that $(X \times Y, \alpha_{1})$ and $(X \times Y, \alpha_{2})$ are both joinings such that $\mathrm{pr}_{X}$ is relatively measure-preserving.  Fix $f \in L^{\infty}(Y,\eta)$ and let $F_{1}(x) = D_{\mathrm{pr}_{X}}^{1}(x)(f\circ \mathrm{pr}_{Y})$ and $F_{2}(x) = D_{\mathrm{pr}_{X}}^{2}(x)(f\circ \mathrm{pr}_{Y})$ where $D_{\mathrm{pr}_{X}}^{j}$ is the disintegration of $\alpha_{j}$ over $\nu$.  Set $F(x) = F_{1}(x) - F_{2}(x)$.  Then $F$ is a bounded Borel function on $X$ and by the above we have that $g\nu(F) = g\nu(F_{1}) - g\nu(F_{2}) = g\eta(f) - g\eta(f) = 0$ for all $g \in G$.  Since $(X,\nu)$ is contractive we also know that $\| F \|_{L^{\infty}(X,\nu)} = \sup_{g} |g\nu(F)| = 0$.  Therefore $F(x) = 0$ almost surely and so $F_{1}(x) = F_{2}(x)$ almost surely.  As this holds for all $f \in L^{\infty}(Y,\eta)$ we conclude that
\[
(\mathrm{pr}_{Y})_{*}D_{\mathrm{pr}_{X}}^{1}(x) = (\mathrm{pr}_{Y})_{*}D_{\mathrm{pr}_{X}}^{2}(x)
\]
for almost every $x \in X$.  The conclusion now follows since the measures have the same disintegration.
\end{proof}

\begin{corollary}\label{C:contractiveoppmp}
Let $(X,\nu)$ be a contractive $G$-space and $(Y,\eta)$ be a measure-preserving $G$-space.  Then the independent joining is the only joining of $(X,\nu)$ and $(Y,\eta)$ such that $\mathrm{pr}_{X}$ is relatively measure-preserving.
\end{corollary}
\begin{proof}
Observe that the independent joining $(X \times Y, \nu \times \eta)$ is a joining and that $D_{\mathrm{pr}_{X}}(x) = \delta_{x} \times \eta$.  Since $(Y,\eta)$ is measure-preserving,
\[
D_{\mathrm{pr}_{X}}(gx) = \delta_{gx} \times \eta = (g\delta_{x}) \times \eta = (g\delta_{x}) \times (g\eta) = g(\delta_{x} \times \eta) = g D_{\mathrm{pr}_{X}}(x)
\]
so $\mathrm{pr}_{X}$ is relatively measure-preserving.  By the previous theorem then the independent joining is the unique such joining.
\end{proof}

\begin{corollary}\label{C:isoprod}
Let $(X,\nu)$ be a $G$-space such that $\pi : (X,\nu) \to (Y,\eta)$ is a relatively measure-preserving $G$-map of $G$-spaces and $\varphi : (X,\nu) \to (Z,\zeta)$ is a relatively contractive $G$-map of $G$-spaces where $(Y,\eta)$ is a contractive $G$-space and $(Z,\zeta)$ is a measure-preserving $G$-space.  Then $(X,\nu)$ is $G$-isomorphic to $(Y \times Z, \eta \times \zeta)$.
\end{corollary}
\begin{proof}
By Corollary \ref{C:orthomaps} the map $\pi \times \varphi$ is a $G$-isomorphism of $(X,\nu)$ with $(Y \times Z, (\pi\times\varphi)_{*}\nu)$.  Now $(\mathrm{pr}_{Y})_{*}(\pi\times\varphi)_{*}\nu = \pi_{*}\nu = \eta$ and likewise $(\mathrm{pr}_{Z})_{*}(\pi\times\varphi)_{*}\nu = \zeta$ so $(\pi\times\varphi)_{*}\nu$ is a joining of $(Y,\eta)$ and $(Z,\zeta)$.  Since $\pi$ is relatively measure-preserving and $\pi = \mathrm{pr}_{X} \circ (\pi \times \varphi)$ we have that $\mathrm{pr}_{X}$ is relatively measure-preserving.  The previous corollary then says that it is the independent joining.
\end{proof}

\begin{corollary}
Let $(X,\nu)$ be a contractive $G$-space and $\pi : (X,\nu) \to (Y,\eta)$ a $G$-map of $G$-spaces.  Then the only joining of $(X,\nu)$ and $(Y,\eta)$ such that the projection to $X$ is relatively measure-preserving is the joining $(X\times Y,\tilde{\pi}_{*}\nu)$ where $\tilde{\pi}(x) = (x,\pi(x))$.
\end{corollary}
\begin{proof}
Let $D(x)$ be the disintegration of $\tilde{\pi}_{*}\nu$ over $\nu$.  Then $D(x)$ is supported on $\{ x \} \times Y \cap \mathrm{supp}~\tilde{\pi}_{*}\nu = \{ (x,\pi(x)) \}$.  Therefore $D(x) = \delta_{(x,\pi(x))}$.  So $D(gx) = \delta_{(gx,\pi(gx))} = \delta_{g(x,\pi(x))} = g\delta_{(x,\pi(x))} = gD(x)$.  By the previous theorem this is then the unique joining with projection to $X$ being relatively measure-preserving.
\end{proof}

More generally:
\begin{theorem}\label{T:moregeneral}
Let $(X,\nu)$ be a contractive $G$-space and $\pi : (X,\nu) \to (Y,\eta)$ a $G$-map of $G$-spaces.  Let $\zeta \in P(X\times Y)$ be a joining of $(X,\nu)$ and $(Y,\eta^{\prime})$ for some $\eta^{\prime}$ absolutely continuous with respect to $\eta$ such that the projection to $X$ of $\zeta$ to $\nu$ is relatively measure-preserving.  Then $\zeta = \tilde{\pi}_{*}\nu$ where $\tilde{\pi}(x) = (x,\pi(x))$ and in particular, $\eta^{\prime} = \eta$.
\end{theorem}
\begin{proof}
Let $D$ be the disintegration of $\zeta$ over $\nu$.  Then $D(x) = \delta_{x} \times \zeta_{x}$ for some $\zeta_{x} \in P(Y)$ for almost every $x$.  Note that $D(gx) = gD(x)$  for $g \in G$ since the projection is relatively measure-preserving and therefore $\zeta_{gx} = g\zeta_{x}$ for all $g \in G$.
Let $f \in C(Y)$.  Define $F \in L^{\infty}(X,\nu)$ by
\[
F(x) = f(\pi(x)) - \zeta_{x}(f).
\]

Let $\epsilon > 0$ and take $x_{0} \in X$ such that $|F(x_{0})| > \| F \|_{L^{\infty}(X,\nu)} - \epsilon$.
Since $(X,\nu)$ is contractive, there exists $g_{n} \in G$ such that $g_{n}\nu \to \delta_{x_{0}}$.
Observe that, using that $\zeta_{gx} = g\zeta_{x}$,
\begin{align*}
g_{n}\nu(F) &= \int_{X} f(\pi(g_{n}x)) - \zeta_{g_{n}x}(f)~d\nu(x)
= \int_{X} f(g_{n}\pi(x)) - g_{n}\zeta_{x}(f)~d\nu(x)
= g_{n}\eta(f) - g_{n}\eta^{\prime}(f)
\end{align*}
since $\int_{X} \zeta_{x}~d\nu(x) = \eta^{\prime}$.

Now $\eta^{\prime}$ is absolutely continuous with respect to $\eta$ and $g_{n}\eta = \pi_{*}g_{n}\nu \to \pi_{*}\delta_{x_{0}} = \delta_{\pi(x_{0})}$.  Since $(Y,\eta)$ is contractive, being a factor of a contractive space, by Corollary \ref{C:contractiveuniqueCS} (the proof of which goes through even when $\eta^{\prime}$ is only absolutely continuous with respect to and not necessarily in the same measure class as $\eta$), $g_{n}\eta{\prime} \to \delta_{\pi(x_{0})}$ also.  Therefore
\[
g_{n}\nu(F) = g_{n}\eta(f) - g_{n}\eta^{\prime}(f) \to f(\pi(x_{0})) - f(\pi(x_{0})) = 0
\]
since $f \in C(Y)$.  So we have that $\| F \| < \epsilon$.  This holds for all $\epsilon > 0$ so $F(x) = 0$ almost surely.  As this holds for all $f \in C(Y)$ we then have that $\zeta_{x} = \delta_{\pi(x)}$ almost surely.  This means that $D(x) = \delta_{x} \times \delta_{\pi(x)} = \delta_{\tilde{\pi}(x)}$ almost surely so $\zeta = \tilde{\pi}_{*}\nu$ as claimed.  Since $\mathrm{proj}_{Y}~\tilde{\pi}_{*}\nu = \pi_{*}\nu = \eta$, then $\eta^{\prime} = \mathrm{proj}_{Y}~\zeta = \eta$.
\end{proof}

We also obtain a special case of a result of Furstenberg and Glasner.  Proposition 3.1 in \cite{FG10} states that there is a unique stationary joining between a $G$-boundary and an arbitrary $G$-space; we obtain another proof of this fact when the $G$-space is measure-preserving:
\begin{corollary}[Furstenberg-Glasner \cite{FG10}]
Let $G$ be a group and $\mu \in P(G)$ a probability measure on $G$.  Let $(B,\beta)$ be the $(G,\mu)$-boundary and $(X,\nu)$ a measure-preserving $G$-space.  Then the only joining $(B\times X,\alpha)$ of $(B,\beta)$ and $(X,\nu)$ such that $\mu * \alpha = \alpha$ is the independent joining.
\end{corollary}
\begin{proof}
Let $\pi : G^{\mathbb{N}} \to B$ be the boundary map (see \cite{BS04} section 2), meaning that $\beta_{\omega} = \lim_{n} \omega_{1}\cdots\omega_{n}\beta = \delta_{\pi(\omega)}$ $\mu^{\mathbb{N}}$-almost surely and $\pi_{*}\mu^{\mathbb{N}} = \beta$.  Since $\alpha$ is $\mu$-stationary, $\alpha = \int \alpha_{\omega}~d\mu^{\mathbb{N}}(\omega)$.  Now $(\mathrm{proj}_{B})_{*}\alpha_{\omega} = \beta_{\omega} = \delta_{\pi(\omega)}$ and $(\mathrm{proj}_{X})_{*}\alpha_{\omega} = \nu_{\omega} = \nu$ since $(X,\nu)$ is measure-preserving.  Therefore $\alpha_{\omega} = \delta_{\pi(\omega)} \times \nu$ and since $\pi_{*}\mu^{\mathbb{N}} = \beta$ then the disintegration of $\alpha$ over $\beta$ is $D(b) = \delta_{b} \times \nu$ which is $G$-equivariant.  Hence the projection to $B$ is relatively measure-preserving so the claim follows by the previous corollaries.
\end{proof}

\subsection{Relatively Contractive Maps and Finite Index Subgroups}

Relative contractiveness is not affected by passage to finite index subgroups.  The proof of the following fact can be found in \cite{CP14} and is a relatively easy exercise for the reader.

\begin{theorem}\label{T:rcfi}
Let $G$ be a locally compact second countable group and $H < G$ be a finite index subgroup.  Let $\pi : (X,\nu) \to (Y,\eta)$ be a relatively contractive $G$-map of ergodic $G$-spaces.  Then, restricting the actions to $H$ makes $\pi$ a relatively contractive $H$-map.
\end{theorem}

\subsection{Contractive Actions and Lattices}

The following is a generalization of Proposition 3.7 in \cite{CS14} (which shows the same result only for Poisson boundaries):
\begin{theorem}\label{T:contractivelattice}
Let $G$ be a locally compact second countable group and $\Gamma < G$ a lattice.  Let $(X,\nu)$ be a contractive $(G,\mu)$-space (meaning that $\mu * \nu = \nu$) for some symmetric $\mu \in P(G)$ such that the support of $\mu$ generates $G$.  Then the restriction of the $G$-action to $\Gamma$ makes $(X,\nu)$ a contractive $\Gamma$-space.
\end{theorem}

\begin{lemma}\label{L:contractivestationary}
Let $G$ be a locally compact second countable group and let $(X,\nu)$ be a contractive $(G,\mu)$-space for some $\mu \in P(G)$ such that the support of $\mu$ generates $G$.  Let $A \subseteq X$ be a measurable set with $\nu(A) > 0$.  Then for every $\epsilon > 0$,
\[
\mu^{\mathbb{N}}(\{ (\omega_{1},\omega_{2},\ldots) \in G^{\mathbb{N}} : \lim_{n \to \infty} \nu(\omega_{n}^{-1}\cdots\omega_{1}^{-1} A) > 1 - \epsilon \}) > 0.
\]
\end{lemma}
\begin{proof}
Let $\varphi \in L^{\infty}(G)$ be defined by $\varphi(g) = \nu(g^{-1} A)$.  Then $\varphi$ is a $\mu$-harmonic nonnegative bounded function on $G$ since $\mu * \nu = \nu$.  As $G \actson (X,\nu)$ is contractive, $\| \varphi \|_{\infty} = \| \bbone_{A} \|_{\infty} = 1$ (because the map $L^{\infty}(X,\nu) \to L^{\infty}(G)$ is an isometry).  Define the function $f \in L^{\infty}(G^{\mathbb{N}})$ by
\[
f(\omega_{1},\omega_{2},\ldots) = \lim_{n \to \infty} \varphi(\omega_{1}\cdots\omega_{n})
\]
which exists $\mu^{\mathbb{N}}$-almost everywhere by the Martingale Convergence Theorem (in fact $f$ descends to an $L^{\infty}$-function on the Poisson boundary of $(G,\mu)$).  Then $f \geq 0$ and $\| f \|_{\infty} = \| \varphi \|_{\infty} = 1$ since the mapping between $L^{\infty}$ of the Poisson boundary and the harmonic functions on $G$ is an isometry.  Let $\epsilon > 0$.  Then
\[
\mu^{\mathbb{N}}(\{ (\omega_{1},\omega_{2},\ldots) \in G^{\mathbb{N}} : f(\omega_{1},\ldots) > 1 - \epsilon \}) > 0
\]
since $\| f \|_{\infty} = 1$.  Since
\[
f(\omega_{1},\ldots) = \lim_{n\to\infty} \varphi(\omega_{1}\cdots\omega_{n}) = \lim_{n\to\infty} \nu(\omega_{n}^{-1}\cdots\omega_{1}^{-1} A),
\]
this completes the proof.
\end{proof}

\begin{proof}[Proof of Theorem \ref{T:contractivelattice}]
Let $m$ be the invariant (Haar) probability measure on $G / \Gamma$.  Let $K_{0}$ be a bounded open subset of $G$.  Set $K = K_{0}\Gamma \subseteq G / \Gamma$.  Then $m(K) > 0$ since $K_{0}$ is open.  By the Random Ergodic Theorem (due to Kifer\cite{kifer} in general and Kakutani \cite{kakutani} in the measure-preserving case), for $m$-almost every $z \in G / \Gamma$ and $\mu^{\mathbb{N}}$-almost every $(\omega_{1},\omega_{2},\ldots)$ it holds that
\[
\lim_{N \to \infty} \frac{1}{N}\sum_{n=1}^{N} \bbone_{K}(\omega_{n}\cdots\omega_{1}z) = m(K) > 0.
\]
Pick $z \in G / \Gamma$ such that the above holds $\mu^{\mathbb{N}}$-almost everywhere.  Then $\omega_{n}\cdots\omega_{1}z \in K$ infinitely often $\mu^{\mathbb{N}}$-almost surely and so, as $\mu$ is symmetric, $\omega_{n}^{-1}\cdots\omega_{1}^{-1}z \in K$ infinitely often $\mu^{\mathbb{N}}$-almost surely.

Let $z_{0}$ be a representative of $z$ in $G$.  Let $B \subseteq X$ be a measurable set with $\nu(B) > 0$.  Set $A = z_{0}B$.  Fix $\epsilon > 0$.  Then $\nu(A) > 0$ since $\nu$ is quasi-invariant and so, by Lemma \ref{L:contractivestationary},
\[
\mu^{\mathbb{N}}(\{ (\omega_{1}, \omega_{2}, \ldots) \in G^{\mathbb{N}} : \lim_{n \to \infty} \nu(\omega_{n}^{-1}\cdots\omega_{1}^{-1} A) > 1 - \epsilon \}) > 0.
\]
As the intersection of a positive measure set with a measure one set is nonempty, there then exists $(\omega_{1},\omega_{2},\ldots)$ such that $\omega_{n}^{-1}\cdots\omega_{1}^{-1}z \in K$ infinitely often and $\lim_{n \to \infty} \nu(\omega_{n}^{-1}\cdots\omega_{1}^{-1} A) > 1 - \epsilon$.  Hence there exists $n$ such that $\omega_{n}^{-1}\cdots\omega_{1}^{-1}z \in K$ and $\nu(\omega_{n}^{-1}\cdots\omega_{1}^{-1} A) > 1 - 2\epsilon$.

\vspace{1pt}
Observe that $\omega_{n}^{-1}\cdots\omega_{1}^{-1}z_{0} \in K_{0}\Gamma$ since $z_{0} \in z\Gamma$ and $\omega_{n}^{-1}\cdots\omega_{1}^{-1}z \in K = K_{0}\Gamma$.  Write $\omega_{n}^{-1}\cdots\omega_{1}^{-1}z_{0} = k\gamma$ for some $k \in K_{0}$ and $\gamma \in \Gamma$.  Then
\[
1 - 2\epsilon < \nu(\omega_{n}^{-1}\cdots\omega_{1}^{-1} A) = \nu(\omega_{n}^{-1}\cdots\omega_{1}^{-1}z_{0} B) = \nu(k\gamma B).
\]
As this holds for all $\epsilon > 0$, there then exists sequences $\{ k_{n} \}$ in $K_{0}$ and $\{ \gamma_{n} \}$ in $\Gamma$ such that $\nu(k_{n}\gamma_{n} B) \to 1$.  As $K$ is bounded, $\overline{K}$ is compact so there exists a subsequence $\{ k_{n_{j}} \}$ such that $k_{n_{j}} \to k_{\infty} \in G$.

Set $C = X \setminus B$ and set $C_{j} = k_{n_{j}} \gamma_{n_{j}} C$.  Then $\nu(C_{j}) \to 0$.  Since $k_{n_{j}}^{-1} \to k_{\infty}^{-1}$ and $\nu(C_{j}) \to 0$, by the continuity of the $G$-action on $L^{1}(X,\nu)$, it follows that $\nu(k_{n_{j}}^{-1}C_{j}) \to 0$.  Therefore $\nu(\gamma_{n_{j}} C) \to 0$ meaning $\nu(\gamma_{n_{j}} B) \to 1$.  As $B$ was an arbitrary measurable set of positive measure, this shows that the $\Gamma$-action on $(X,\nu)$ is contractive.
\end{proof}

\subsection{Inducing Relatively Contractive Maps}

\begin{theorem}
Let $\Gamma < G$ be a lattice in a locally compact second countable group.  Let $\pi : (X,\nu) \to (Y,\eta)$ be a $\Gamma$-map of $\Gamma$-spaces and let $\Pi : G \times_{\Gamma} X \to G \times_{\Gamma} Y$ be the induced $G$-map of $G$-spaces.  Then $\pi$ is a relatively contractive $\Gamma$-map if and only if $\Pi$ is a relatively contractive $G$-map.
\end{theorem}
\begin{proof}
Assume first that $\Pi$ is relatively contractive.
Fix a fundamental domain $(F,m)$ for $G / \Gamma$ as in the induced action construction and let $\alpha : G \times F \to \Gamma$ be the associated cocycle for the $G$-action on $F \times X$.  Let $\Phi : (F \times X, m \times \nu) \to (F \times Y, m \times \eta)$ by $\Phi = \mathrm{id} \times \pi$.  Then $\Phi$ is isomorphic to $\Pi$ over the canonical isomorphisms $G \times_{\Gamma} X \simeq F \times X$ and $G \times_{\Gamma} Y \simeq F \times Y$ so $\Phi$ is relatively contractive.  Consider the disintegration map $D_{\Phi} : F \times Y \to P(F \times X)$.  Observe that for $(f,y) \in F \times Y$
\[
D_{\Phi}(f,y) = \delta_{f} \times D_{\pi}(y)
\]
since $\Phi = \mathrm{id} \times \pi$ and all the spaces have the product measure.  Now consider the conjugates of the disintegration map: for $g \in G$ and $(f,y) \in F \times Y$,
\begin{align*}
D_{\Phi}^{(g)}(f,y) &= g^{-1}D_{\Phi}(g(f,y)) = g^{-1}D_{\Phi}(gf\alpha(g,f),\alpha(g,f)^{-1}y) \\
&= g^{-1}(\delta_{gf\alpha(g,f)} \times D_{\pi}(\alpha(g,f)^{-1}y)) \\
&= \delta_{f} \times \alpha(g^{-1},gf\alpha(g,f))^{-1}D_{\pi}(\alpha(g,f)^{-1}y) \\
&= \delta_{f} \times D_{\pi}^{(\alpha(g,f)^{-1})}(y).
\end{align*}
Now take $r \in L^{\infty}(X,\nu)$ and define $q(f,x) = r(x)$.  Then for $m\times\eta$-almost every $(f,y)$
\[
\| q \|_{L^{\infty}(F \times X,D_{\Phi}(f,y))} = \| r \|_{L^{\infty}(X,D_{\pi}(y))}
\]
and since $\Phi$ is relatively contractive, for $m \times \eta$-almost every $(f,y)$ there exists $g_{n} \in G$ such that
\[
D_{\Phi}^{(g_{n})}(f,y)(q) \to \| q \|_{L^{\infty}(F \times X,D_{\Phi}(f,y))}.
\]
Therefore
\[
\delta_{f} \times D_{\pi}^{(\alpha(g_{n},f)^{-1})}(y)(q) \to \| r \|_{L^{\infty}(X,D_{\pi}(y))}
\]
and by construction of $q$ then
\[
D_{\pi}^{(\alpha(g_{n},f)^{-1})}(y)(r) \to \| r \|_{L^{\infty}(X,D_{\pi}(y))}.
\]
Hence for $\eta$-almost every $y$ there exists a sequence $\gamma_{n} = \alpha(g_{n},f)^{-1} \in \Gamma$ (outside of possibly a measure zero set, which $f$ is chosen is irrelevant) such that
\[
D_{\pi}^{(\gamma_{n})}(r) \to \| r \|_{L^{\infty}(X,D_{\pi}(y))}
\]
which means that $\pi$ is relatively contractive.

Now assume that $\pi$ is relatively contractive.  Let $x \in X$ and $f \in F$ and set $y = \pi(x)$.  As above,
\[
D_{\Phi}^{(g)}(f,y) = \delta_{f} \times D_{\pi}^{(\alpha(g,f)^{-1})}(y).
\]
Since $\pi$ is relatively contractive, there exists $\{ \gamma_{n} \}$ such that $D_{\pi}^{(\gamma_{n})}(y) \to \delta_{x}$.  Set $g_{n} = f\gamma_{n}f^{-1}$.  Then $\alpha(g_{n},f) = \gamma_{n}^{-1}$ and so
\[
D_{\Phi}^{(g_{n})}(f,y) = \delta_{f} \times D_{\pi}^{(\gamma_{n})}(y) \to \delta_{(f,x)}
\]
meaning that $\Pi$ is relatively contractive.
\end{proof}

\subsection{Relative Joinings Over Relatively Contractive Maps}

Relatively contractive maps were introduced in \cite{CP14} and used to show that any joining between a contractive space and a measure-preserving space such that the projection to the contractive space is relatively measure-preserving is necessarily the independent joining.  We generalize this fact to the case of relative joinings and obtain an analogous result.

\begin{theorem}
Let $(X,\nu)$ and $(Y,\eta)$ be $G$-spaces with a common $G$-quotient $(Z,\zeta)$ such that $\varphi : (Y,\eta) \to (Z,\zeta)$ is relatively contractive and $\pi : (X,\nu) \to (Z,\zeta)$ is a $G$-map.  Then there exists at most one relative joining of $(X,\nu)$ and $(Y,\eta)$ over $(Z,\zeta)$ such that the projection to $(Y,\eta)$ is relatively measure-preserving.
\end{theorem}
\begin{proof}
For convenience, write $W = X \times Y$.
Let $\rho$ be a relative joining of $(X,\nu)$ and $(Y,\eta)$ over $(Z,\zeta)$ such that $\varphi : (Y,\eta) \to (Z,\zeta)$ is relatively contractive, $p_{Y} : (W,\rho) \to (Y,\eta)$ is relatively measure-preserving and $p_{X} : (W,\rho) \to (X,\nu)$ and $\pi : (X,\nu) \to (Z,\zeta)$ are $G$-maps such that $\pi \circ p_{X} = \varphi \circ p_{Y}$ almost everywhere.  Denote by $\psi : (W,\rho) \to (Z,\zeta)$ the composition: $\psi = \pi \circ p_{X} = \varphi \circ p_{Y}$.

Let $z \in Z$ and let $f \in L^{\infty}(\pi^{-1}(z), D_{\pi}(z))$ be arbitrary.  Then $f \circ p_{X} \in L^{\infty}(\psi^{-1}(z), D_{\psi}(z))$ since $D_{\psi}(z) = \int D_{p_{X}}(x)~dD_{\pi}(z)(x)$.  Define 
\[
F(y) = D_{p_{Y}}(y)(f \circ p_{X})
\]
and observe that $F \in L^{\infty}(\varphi^{-1}(z), D_{\varphi}(z))$.

For an arbitrary $g \in G$, using that $p_{Y}$ is relatively measure-preserving,
\begin{align*}
D_{\varphi}^{(g)}(z)(F) &= \int_{\varphi^{-1}(z)} F(y)~dg^{-1}D_{\varphi}(gz) \\
&= \int_{\varphi^{-1}(gz)} F(g^{-1}y)~dD_{\varphi}(gz) \\
&= \int_{\varphi^{-1}(gz)} \int_{p_{Y}^{-1}(g^{-1}y)} f(p_{X}(w))~dD_{p_{Y}}(g^{-1}y)(w)~dD_{\varphi}(gz)(y) \\
&= \int_{\varphi^{-1}(gz)} \int_{p_{Y}^{-1}(g^{-1}y)} f(p_{X}(w))~dg^{-1}D_{p_{Y}}(y)(w)~dD_{\varphi}(gz)(y) \\
&= \int_{\varphi^{-1}(gz)} \int_{p_{Y}^{-1}(y)} f(p_{X}(g^{-1}w))~dD_{p_{Y}}(y)(w)~dD_{\varphi}(gz)(y) \\
&= \int_{\varphi^{-1}(gz)} \int_{p_{Y}^{-1}(y)} f(g^{-1} p_{X}(w))~dD_{p_{Y}}(y)(w)~dD_{\varphi}(gz)(y)
\end{align*}
Now $\int_{\varphi^{-1}(gz)} D_{p_{Y}}(y)~dD_{\varphi}(gz)(y) = D_{\psi}(gz)$ and therefore
\begin{align*}
D_{\varphi}^{(g)}(z)(F) &= \int_{\psi^{-1}(gz)} f(g^{-1}p_{X}(w))~dD_{\psi}(gz)(w)
= \int_{p_{X}(\psi^{-1}(gz))} f(g^{-1}x)~d((p_{X})_{*}D_{\psi}(gz))(x) \\
&= \int_{\pi^{-1}(gz)} f(g^{-1}x)~dD_{\pi}(gz)(x)
= D_{\pi}^{(g)}(z)(f).
\end{align*}

Now let $\rho_{1}$ and $\rho_{2}$ both be relative joinings over $(Z,\zeta)$.  Since $\varphi$ is relatively contractive, there is a measure one set of $z \in Z$ such that for all $F \in L^{\infty}(\varphi^{-1}(z),D_{\varphi}(z))$, we have that $\sup_{g \in G} |D_{\varphi}^{(g)}(F)| = \| F \|$.  Fix $z$ in this measure one set.

Let $f \in L^{\infty}(\pi^{-1}(z),D_{\pi}(z))$ be arbitrary.  Let $D_{p_{Y}}^{j}$ and $D_{\psi}^{j}$ for $j = 1,2$ denote the disintegrations of $\rho_{1}$ and $\rho_{2}$ over $\eta$ and $\zeta$, respectively.
Define, for $j = 1,2$,
\[
F_{j}(y) = D_{p_{Y}}^{j}(f \circ p_{X})
\]
and set $F(y) = F_{1}(y) - F_{2}(y)$.  As above, $F \in L^{\infty}(\varphi^{-1}(z),D_{\varphi}(z))$.  Now, by the above, for any $g \in G$,
\[
D_{\varphi}^{(g)}(z)(F_{1}) = D_{\pi}^{(g)}(z)(f) = D_{\varphi}^{(g)}(z)(F_{2})
\]
and therefore $D_{\varphi}^{(g)}(z)(F) = 0$.

Since $z$ is in the measure one set where that map is an isometry, $\| F \| = \sup_{g} |D_{\varphi}^{(g)}(z)(F)| = 0$.  Therefore $F = 0$ almost everywhere.  As this holds for all $f \in L^{\infty}(\pi^{-1}(z),D_{\pi}(z))$, we conclude that $D_{\varphi}^{1}(y) = D_{\varphi}^{2}(y)$ for $D_{\varphi}(z)$-almost-every $y \in \varphi^{-1}(z)$.

Now let $f \in L^{\infty}(\psi^{-1}(z),D_{\psi}(z))$ be arbitrary and observe that
\begin{align*}
D_{\psi}^{j}(z)(f) &= \int_{\psi^{-1}(z)} f(x,y)~dD_{\psi}^{j}(z)(x,y)
= \int_{\varphi^{-1}(z)} \int_{p_{Y}^{-1}(y)} f(x,y)~dD_{p_{Y}}^{j}(y)(x)~dD_{\varphi}(z)(y).
\end{align*}
Since $D_{p_{Y}}^{1}(y) = D_{p_{Y}}^{2}(y)$ for $D_{\varphi}(z)$-almost every $y$,
\[
D_{\psi}^{1}(z)(f) = D_{\psi}^{2}(z)(f).
\]
This holds for all $f \in L^{\infty}(\psi^{-1}(z),D_{\psi}(z))$ and so $D_{\psi}^{1}(z) = D_{\psi}^{2}(z)$.

Since the above holds for all $z$ in a measure one set,
\[
\rho_{1} = \int_{Z} D_{\psi}^{1}(z)~d\zeta(z) = \int_{Z} D_{\psi}^{2}(z)~d\zeta(z) = \rho_{2}.
\]
\end{proof}

\begin{corollary}
Let $(X,\nu)$ and $(Y,\eta)$ be $G$-spaces with a common $G$-quotient $(Z,\zeta)$ such that $\varphi : (Y,\eta) \to (Z,\zeta)$ is relatively contractive and $\pi : (X,\nu) \to (Z,\zeta)$ is relatively measure-preserving.  Then the only relative joining of $(X,\nu)$ and $(Y,\eta)$ over $(Z,\zeta)$ such that the projection to $(Y,\eta)$ is relatively measure-preserving is the independent relative joining.
\end{corollary}
\begin{proof}
By the previous theorem, we need only show that the independent relative joining $\rho = \int D_{\pi} \times D_{\varphi}~d\zeta$ is a relative joining such that the projection to $(Y,\eta)$ is relatively measure-preserving.  Let $D_{p_{Y}}$ be the disintegration of $\rho$ over $\eta$.  Observe that $p_{Y}^{-1}(y) = \pi^{-1}(\varphi(y)) \times \{ y \}$ and that the support of $D_{\pi}(\varphi(y)) \times \delta_{y}$ is the same.  Now
\begin{align*}
\int_{Y} D_{\pi}(\varphi(y)) \times \delta_{y}~d\eta(y) &= \int_{Z} \int_{Y} D_{\pi}(z) \times \delta_{y}~dD_{\varphi}(z)(y)~d\eta(y) \\
&= \int_{Z} D_{\pi}(z) \times D_{\varphi}(z)~d\zeta(z) = \rho
\end{align*}
so by uniqueness, $D_{p_{Y}}(y) = D_{\pi}(\varphi(y)) \times \delta_{y}$ almost everywhere.  
Then, using that $\pi$ is relatively measure-preserving,
\[
D_{p_{Y}}(gy) = D_{\pi}(\varphi(gy)) \times \delta_{gy} = gD_{\pi}(\varphi(y)) \times g \delta_{y} = g D_{p_{Y}}(y)
\]
so $p_{Y}$ is relatively measure-preserving.  By the previous theorem, $\rho$ is then the unique relative joining.
\end{proof}

\begin{corollary}\label{C:reljoinunique}
Let $G$ be a locally compact second countable group and let $(X,\nu), (Y,\eta), (Z,\zeta)$ and $(W,\rho)$ be $G$-spaces such that the following diagram of $G$-maps commutes:
\begin{diagram}
(W,\rho)		&\rTo^{\psi}	&(X,\nu)\\
\dTo^{\tau}	&			&\dTo^{\pi}\\
(Y,\eta)		&\rTo^{\varphi}	&(Z,\zeta)
\end{diagram}
If $\tau$ and $\pi$ are relatively measure-preserving and $\psi$ and $\varphi$ are relatively contractive then $(W,\rho)$ is $G$-isomorphic to the independent relative joining of $(X,\nu)$ and $(Y,\eta)$ over $(Z,\zeta)$.
\end{corollary}
\begin{proof}
Consider the map $p : W \to X \times Y$ by $p(w) = (\psi(w), \tau(w))$.  Then $p_{*}\rho$ is a relative joining of $(X,\nu)$ and $(Y,\eta)$ over $(Z,\zeta)$.  Let $p_{X} : X \times Y \to X$ and $p_{Y} : X \times Y \to Y$ be the natural projections and observe that the following diagram commutes:
\begin{diagram}
(W,\rho)		&\rTo^{p}	&(X \times Y,p_{*}\rho)	&\rTo^{p_{X}}	&(X,\nu)\\
			&		&\dTo^{p_{Y}}			&			&\dTo^{\pi}\\
			&		&(Y,\eta)				&\rTo^{\varphi}	&(Z,\zeta)
\end{diagram}
since $p_{X} \circ p = \psi$ and $p_{Y} \circ p = \tau$.

Now $\psi$ is relatively contractive so $p$ and $p_{X}$ are relatively contractive and likewise $\tau$ being relatively measure-preserving implies $p$ and $p_{Y}$ are relatively measure-preserving.  Therefore $p$ is an isomorphism (Proposition \ref{P:bothrels}).  Since $\varphi$ is relatively contractive and $p_{Y}$ is relatively measure-preserving and $\pi$ is relatively measure-preserving, the previous corollary says that $p_{*}\rho$ is the independent relative joining.
\end{proof}

\section{The Factor Theorems}

The results in the previous section on the uniqueness of relatively contractive maps and their (lack of) joinings with measure-preserving systems lead to the so-called factor theorems that are the main tool in the rigidity theorems presented in the final section.

\subsection{The Intermediate Contractive Factor Theorem}

The first factor theorem we present appears in \cite{CP14} and is a generalization of the factor theorem for contractive actions in \cite{CS14}.

\begin{theorem}\label{T:contractiveIFT}
Let $\Gamma < G$ be a lattice in a locally compact second countable group and let $\Lambda$ contain and commensurate $\Gamma$ (commensurate meaning that $\Gamma \cap \lambda\Gamma\lambda^{-1}$ has finite index in $\Gamma$ for each $\lambda \in \Lambda$) and be dense in $G$.  Let $(X,\nu)$ be a contractive $(G,\mu)$-space, meaning that $\mu * \nu = \nu$, (for some $\mu \in P(G)$ such that the support of $\mu$ generates $G$) and $(Y,\eta)$ be a measure-preserving $G$-space.  Let $\pi : (X \times Y, \nu \times \eta) \to (Y,\eta)$ be the natural projection map from the product space with the diagonal action.
Let $(Z,\zeta)$ be a $\Lambda$-space such that there exist $\Gamma$-maps $\varphi : (X\times Y,\nu\times\eta) \to (Z,\zeta)$ and $\rho : (Z,\zeta) \to (Y,\eta)$ with $\rho \circ \varphi = \pi$.  Then $\varphi$ and $\rho$ are $\Lambda$-maps and $(Z,\zeta)$ is $\Lambda$-isomorphic to a $G$-space and over this isomorphism the maps $\varphi$ and $\rho$ become $G$-maps.
\end{theorem}
\begin{proof}
Write $(W,\rho) = (X \times Y, \nu \times \eta)$.
Fix $\lambda \in \Lambda$.  Define the maps $\varphi_{\lambda} : W \to Z$ and $\rho_{\lambda} : Z \to Y$ by $\varphi_{\lambda}(w) = \lambda^{-1}\varphi(\lambda w)$ and $\rho_{\lambda}(z) = \lambda^{-1}\rho(\lambda z)$.  Then $\rho_{\lambda} \circ \varphi_{\lambda}(w) = \lambda^{-1} \rho(\lambda \lambda^{-1}\varphi(\lambda w)) = \lambda^{-1}\rho(\varphi(\lambda w)) = \lambda^{-1}\pi(\lambda w) = \pi(w)$ since $\pi$ is $\Lambda$-equivariant.  Let $\Gamma_{0} = \Gamma \cap \lambda^{-1}\Gamma\lambda$.  Then for $\gamma_{0} \in \Gamma_{0}$, write $\gamma_{0} = \lambda^{-1}\gamma\lambda$ for some $\gamma \in \Gamma$ and we see that $\varphi_{\lambda}(\gamma_{0}w) = \lambda^{-1}\varphi(\lambda\gamma_{0}w) = \lambda^{-1}\varphi(\gamma \lambda w) = \lambda^{-1}\gamma\varphi(\lambda w) = \gamma_{0} \lambda^{-1}\varphi(\lambda w) = \gamma_{0} \varphi_{\lambda}(w)$ meaning that $\varphi_{\lambda}$ is $\Gamma_{0}$-equivariant.  Likewise $\rho_{\lambda}$ is $\Gamma_{0}$-equivariant.  Hence $\varphi$, $\varphi_{\lambda}$, $\rho$ and $\rho_{\lambda}$ are all $\Gamma_{0}$-equivariant.  

Since $(X,\nu)$ is a contractive $(G,\mu)$-space and $\Gamma_{0}$ is a lattice in $G$, by Theorem \ref{T:contractivelattice}, $(X,\nu)$ is a contractive $\Gamma_{0}$-space.  By Theorem \ref{T:relcontractiveunique} applied to $\Gamma$, we can conclude that $\varphi_{\lambda} = \varphi$ and that $\rho_{\lambda} = \rho$ provided we can show that the disintegration measures $D_{\rho}(y)$ and $D_{\rho_{\lambda}}(y)$ are in the same measure class for almost every $y$.  Assuming this for the moment, we then conclude that $\varphi$ is $\Lambda$-equivariant since $\varphi_{\lambda} = \varphi$ for each $\lambda$.  The $\sigma$-algebra of pullbacks of measurable functions on $(Z,\zeta)$ form a $\Lambda$-invariant sub-$\sigma$-algebra of $L^{\infty}(W,\rho)$ which is therefore also $G$-invariant (because $\Lambda$ is dense in $G$) and so $(Z,\zeta)$ has a point realization as a $G$-space \cite{Ma62} and likewise $\varphi$ and $\rho$ as $G$-maps.

It remains only to show that the disintegration measures have the required property.  First note that $D_{\rho}(y) = \varphi_{*}D_{\rho\circ\varphi}(y)$ by the uniqueness of the disintegration measure and likewise that $D_{\rho_{\lambda}}(y) = (\varphi_{\lambda})_{*}D_{\rho_{\lambda}\circ\varphi_{\lambda}}(y) = \lambda^{-1}\varphi_{*}\lambda D_{\rho\circ\varphi}(y) = \lambda^{-1}\varphi_{*}D_{\rho\circ\varphi}^{(\lambda^{-1})}(\lambda y)$.  Now $\rho\circ\varphi = \pi$ is a $\Lambda$-map so $D_{\rho\circ\varphi}^{(\lambda^{-1})}(\lambda y)$ is in the same measure class as $D_{\rho\circ\varphi}(\lambda y)$.  Therefore $D_{\rho_{\lambda}}(y)$ is in the same measure class as $\lambda^{-1}\varphi_{*}D_{\rho\circ\varphi}(\lambda y) = \lambda^{-1} D_{\rho}(\lambda y)$.  Now $\lambda^{-1} D_{\rho}(\lambda y)$ disintegrates $\lambda^{-1}\zeta$ over $\lambda^{-1}\eta$ via $\rho$ and $\lambda^{-1}\zeta$ is in the same measure class as $\zeta$ since $(Z,\zeta)$ is a $\Lambda$-space.  Therefore, by Lemma \ref{L:abscontmeas}, $\lambda^{-1}D_{\rho}(\lambda y)$ and $D_{\rho}(y)$ are in the same measure class for almost every $y$.  Hence $D_{\rho_{\lambda}}(y)$ and $D_{\rho}(y)$ are in the same measure class for almost every $y$ as needed.
\end{proof}

\subsection{The Intermediate Contractive Factor Theorem for Products}

The second factor theorem we present is a strengthening of the Bader-Shalom Intermediate Factor Theorem \cite{BS04} that first appeared in \cite{Cre13}:
\begin{theorem}\label{T:IFT}
Let $G = G_{1} \times G_{2}$ be a product of two locally compact second countable groups and let $\mu_{j} \in P(G_{j})$ be admissible probability measures for $j=1,2$.  Set $\mu = \mu_{1} \times \mu_{2}$.

Let $(B,\beta)$ be the Poisson boundary for $(G,\mu)$ and let $(X,\nu)$ be a measure-preserving $G$-space.  Let $(W,\rho)$ be a $G$-space such that there exist $G$-maps $\pi : (B \times X, \beta\times\nu) \to (W,\rho)$ and $\varphi : (W,\rho) \to (X,\nu)$ with $\varphi \circ \pi$ being the natural projection to $X$.

Let $(W_{1},\rho_{1})$ be the space of $G_{2}$-ergodic components of $(W,\rho)$ and let $(W_{2},\rho_{2})$ be the space of $G_{1}$-ergodic components.  Likewise, let $(X_{1},\nu_{1})$ and $(X_{2},\nu_{2})$ be the ergodic components of $(X,\nu)$ for $G_{2}$ and $G_{1}$, respectively.

Then $(W,\rho)$ is $G$-isomorphic to the independent relative joining of $(W_{1},\rho_{1}) \times (W_{2},\rho_{2})$ and $(X,\nu)$ over $(X_{1},\nu_{1}) \times (X_{2},\nu_{2})$.
\end{theorem}

We opt to omit the proof as it involves both the relatively contractive maps and their relationship to joinings and also results due to Bader and Shalom \cite{BS04} on the nature of the ergodic decomposition of spaces on which products of groups acts (which fall outside our scope).  The reader is referred to \cite{Cre13} for a detailed proof.

\section{Rigidity of Actions of Lattices}

To conclude our exposition, we present now the main results of \cite{CP14} and \cite{Cre13}, all of which rely in crucial fashion on the factor theorems developed above.  These rigidity results are the main application of contractive spaces and relatively contractive maps and were the motivation for the development of these concepts.  While the proofs are beyond the scope of our exposition (and can be found in the respective papers), we stress that the key ingredient in the proofs is the uniqueness property of relatively contractive maps in the form of the factor theorems.

\begin{theorem}[\cite{CP14},\cite{Cre13}]\label{T:app1}
Let $G$ be a semisimple group with trivial center and no compact factors with at least one factor being a connected (real) Lie group with property $(T)$.  Let $\Gamma < G$ be an irreducible lattice (meaning that the projection of $\Gamma$ is dense in every proper normal subgroup of $G$).  Then every measure-preserving action $\Gamma \actson (X,\nu)$ on a nonatomic probability space is essentially free.
\end{theorem}

The proof strategy is to show that any action which is not essentially free is weakly amenable (meaning that the stabilizer subgroups of almost every point are co-amenable in $\Gamma$) which, combined with the (partial) property $(T)$ like behavior of $\Gamma$ forces the action to be atomic.  The key idea is to study the relatively contractive map $B \times X \to X$ where $B$ is a Poisson boundary of $G$ (the map is relatively contractive since $\Gamma$ is a lattice and $B$ is a stationary space).

The intermediate factor theorem guarantees that any $\Gamma$-space $A$ appearing in a chain $B \times X \to A \to X$ combined with the results on joinings with contractive spaces has the property that $A$ is isomorphic to $C \times X$ where $C$ is a quotient of $B$.  Various results of Zimmer (see \cite{Cre13} and \cite{CP14} for details) state that if $\Gamma \actson (X,\nu)$ is not weakly amenable then there exist spaces $A$ not isomorphic to $X$ but sharing the same stabilizer subgroups.  If the stabilizers are not trivial (i.e.~the action is not essentially free), then the factor theorem leads to the conclusion that nontrivial subgroups of $\Gamma$ stabilize points in $C$.  However, it is an easy consequence of the construction of the Poisson boundary that the action on $C$ is always essentially free if $C$ is nontrivial.

The crucial fact in the above strategy is that one obtains a large amount of structural information about such spaces $A$ from the fact that the map is relatively contractive, in particular, enough information to rule out nontrivial intermediate spaces.

We remark that the above theorem implies the Margulis Normal Subgroup Theorem in a direct way: if $N \normal \Gamma$ is nontrivial then the Bernouli action of $\Gamma / N$, treated as a $\Gamma$-space, has stabilizer subgroups precisely equal to $N$ and so the theorem states that in such a case, $N$ must be finite index (the Bernoulli shift must be atomic).

In closing, we also mention that the notion of relatively contractive maps has been extended to the noncommutative setting of operator algebras by the author and J.~Peterson \cite{CPnon}, leading to a sweeping generalization of the normal subgroup theorem:
\begin{theorem}[\cite{CPnon}]\label{T:app2}
Let $G$ be a semisimple group with trivial center and no compact factors with at least one factor being a connected (real) Lie group with property $(T)$ and let $\Gamma < G$ be an irreducible lattice.  Let $\pi : \Gamma \to \mathcal{U}(M)$ be a representation into the unitary group of a finite factor $M$ such that $\pi(\Gamma)^{\prime\prime} = M$.  Then either $M$ is finite-dimensional or $\pi$ extends to an isomorphism of the group von Neumann algebra $L\Gamma \simeq M$.
\end{theorem}

This result gives a form of operator-algebraic superrigidity in the sense that such a lattice $\Gamma$ cannot be ``separated" from its group von Neumann algebra in the same way that the Margulis-Zimmer superrigidty theorem states that it cannot be separated from $G$: if $\varphi : \Gamma \to H$ is a homomorphism into an algebraic group with $\overline{\varphi(\Gamma)}$ noncompact then $\varphi$ extends to an isomorphism of $G$.

The result on operator algebraic superrigidity should be contrasted with the case of amenable groups: if $\Gamma$ and $\Lambda$ are amenable countable groups then $L\Gamma$ is always isomorphic to $L\Lambda$.  In this sense, lattices in semisimple groups are as far from amenable as possible and the superrigidity theorem is a major indication of this.

\dbibliography{references}

\end{document}